\documentclass[oneside,english]{amsart}
\usepackage[T1]{fontenc}
\usepackage{color}
\usepackage{amsthm}
\usepackage{bm}
\usepackage{amstext}
\usepackage{amssymb}
\usepackage[all]{xy}
\usepackage{colortbl}

\makeatletter
\numberwithin{equation}{section} %% Comment out for sequentially-numbered
\numberwithin{figure}{section} %% Comment out for sequentially-numbered
\theoremstyle{plain}
\newtheorem{thm}{Theorem}[section]
  \theoremstyle{definition}
  \newtheorem{defn}[thm]{Definition}
\newtheorem*{ackn}{Acknowledgement}
  \theoremstyle{plain}
  \newtheorem{lem}[thm]{Lemma}
  \theoremstyle{plain}
  \newtheorem{prop}[thm]{Proposition}
  \newtheorem{cor}[thm]{Corollary}
  \theoremstyle{remark}
  \newtheorem*{rem*}{Remark}
 \theoremstyle{definition}
  
  \theoremstyle{remark}
  \newtheorem{rem}[thm]{Remark}

\usepackage{amsmath}
\usepackage{amsfonts}
\usepackage{amsthm}
\usepackage{amssymb}
\usepackage{amscd}
\usepackage[all]{xy}
\usepackage{enumerate}
\usepackage{mathrsfs}
\usepackage{graphicx, bm, color, ulem}

\def\ft#1{{\mathsf #1}}
\def\chow{{\mathscr X}}
\def\hcoY{{\mathscr Y}}
\def\Hes{{\mathscr H}}
\def\UU{{\mathscr U}}
\def\Zpq{{\mathscr Z}}

\def\hchow{{\mathaccent20\chow}}
\def\Prt{{\mathscr P}}
\def\frQ{{\mathfrak Q}}
\def\Lrho{\text{\large{\mbox{$\rho\hskip-2pt$}}}}
\def\tLrho{\text{\large{\mbox{$\tilde\rho\hskip-2pt$}}}}
\def\Lpi{\text{\large{\mbox{$\pi\hskip-2pt$}}}}
\def\Lp{\text{\large{\mbox{$p\hskip-1pt$}}}}

\font\eu=eusm10 at 10pt

\def\eF{\text{\eu F}}
\def\eG{\text{\eu G}}

\input xyimport.tex

\newcommand{\vcorr}[3][1]{%
  \begingroup
    \tabcolsep=.5\tabcolsep
    \sbox0{%
      \begin{tabular}[b]{@{}l}%
        #3%
         \tabularnewline
      \end{tabular}%
    }%
    \settoheight{\dimen0 }{%
      \rotatebox{#2}{%
        \copy0 %
        \kern-\tabcolsep
      }%
    }%
    \rule{0pt}{#1\dimen0}%
    \setlength{\wd0 }{1em}%
    \setlength{\ht0 }{1em}%
    \rotatebox{#2}{\usebox{0}}%
  \endgroup
}

\newenvironment{myitem}{\begin{list}{}{
\setlength{\leftmargin}{25pt}
\setlength{\labelsep}{5pt}
}}{\end{list}}
\newenvironment{myitem2}{\begin{list}{}{
\setlength{\leftmargin}{0.5cm}
\setlength{\itemindent}{-0.3cm}
\setlength{\itemsep}{0cm}
}}{\end{list}}
\makeatother

\usepackage{babel}

\begin{document}
\global\long\def\sA{\mathcal{A}}
 \global\long\def\sB{\mathcal{B}}
 \global\long\def\sC{\mathcal{C}}
 \global\long\def\sD{\mathcal{D}}
 \global\long\def\sE{\mathcal{E}}
 \global\long\def\sF{\eF}
 \global\long\def\sG{\eG}
 \global\long\def\sH{\mathcal{H}}
 \global\long\def\sI{\mathcal{I}}
 \global\long\def\sJ{\mathcal{J}}
 \global\long\def\sK{\mathcal{K}}
 \global\long\def\sL{\mathcal{L}}
 \global\long\def\sN{\mathcal{N}}
 \global\long\def\sM{\mathcal{M}}
 \global\long\def\sO{\mathcal{O}}
 \global\long\def\sP{\mathcal{P}}
 \global\long\def\sS{\mathcal{S}}
 \global\long\def\sR{\mathcal{R}}
 \global\long\def\sQ{\mathcal{Q}}
 \global\long\def\sT{\mathcal{T}}
 \global\long\def\sU{\mathcal{U}}
 \global\long\def\sV{\mathcal{V}}
 \global\long\def\sW{\mathcal{W}}
 \global\long\def\sX{\mathcal{X}}
 \global\long\def\sY{\mathcal{Y}}
 \global\long\def\sZ{\mathcal{Z}}
 \global\long\def\tA{{\widetilde{A}}}
 \global\long\def\mA{\mathbb{A}}
 \global\long\def\mC{\mathbb{C}}
 \global\long\def\mF{\mathbb{F}}
 \global\long\def\mG{\mathbb{G}}
 \global\long\def\G{{\bf G}}
 \global\long\def\mN{\mathbb{N}}
 \global\long\def\mP{\mathbb{P}}
 \global\long\def\mQ{\mathbb{Q}}
 \global\long\def\mZ{\mathbb{Z}}
 \global\long\def\mW{\mathbb{W}}
 \global\long\def\Ima{\mathrm{Im}\,}
 \global\long\def\Ker{\mathrm{Ker}\,}
 \global\long\def\Alb{\mathrm{Alb}\,}
 \global\long\def\ap{\mathrm{ap}}
 \global\long\def\Bs{\mathrm{Bs}\,}
 \global\long\def\Chow{\mathrm{Chow}}
 \global\long\def\CP{\mathrm{CP}}
 \global\long\def\Div{\mathrm{Div}\,}
 \global\long\def\divi{\mathrm{div}\,}
 \global\long\def\expdim{\mathrm{expdim}\,}
 \global\long\def\ord{\mathrm{ord}\,}
 \global\long\def\Aut{\mathrm{Aut}\,}
 \global\long\def\Hilb{\mathrm{Hilb}}
 \global\long\def\Hom{\mathrm{Hom}}
 \global\long\def\id{\mathrm{id}}
 \global\long\def\Ext{\mathrm{Ext}}
 \global\long\def\sHom{\mathcal{H}{\!}om\,}
 \global\long\def\Lie{\mathrm{Lie}\,}
 \global\long\def\mult{\mathrm{mult}}
 \global\long\def\opp{\mathrm{opp}}
 \global\long\def\Pic{\mathrm{Pic}\,}
 \global\long\def\Pf{{\bf Pf}}
 \global\long\def\Sec{\mathrm{Sec}}
 \global\long\def\Spec{\mathrm{Spec}\,}
 \global\long\def\Sym{\mathrm{Sym}}
 \global\long\def\sQpec{\mathcal{S}{\!}pec\,}
 \global\long\def\Proj{\mathrm{Proj}\,}
 \global\long\def\Rhom{{\mathbb{R}\mathcal{H}{\!}om}\,}
 \global\long\def\aw{\mathrm{aw}}
 \global\long\def\exc{\mathrm{exc}\,}
 \global\long\def\emb{\mathrm{emb\text{-}dim}}
 \global\long\def\codim{\mathrm{codim}\,}
 \global\long\def\OG{\mathrm{OG}}
 \global\long\def\pr{\mathrm{pr}}
 \global\long\def\Sing{\mathrm{Sing}\,}
 \global\long\def\Supp{\mathrm{Supp}\,}
 \global\long\def\SL{\mathrm{SL}\,}
 \global\long\def\Reg{\mathrm{Reg}\,}
 \global\long\def\rank{\mathrm{rank}\,}
 \global\long\def\VSP{\mathrm{VSP}\,}
 \global\long\def\B{B}
 \global\long\def\Q{Q}
 \global\long\def\rG{\mathrm{G}}
\global\long\def\rF{\mathrm{F}}
 \global\long\def\nS{\textsc{S}}
\global\long\def\nT{\textsc{T}}
\global\long\def\nU{\textsc{U}}
 \global\long\def\frP{\mathfrak{P}}
\global\long\def\frR{\mathfrak{R}}

%\hfill \textcolor{red}{\today}

%\title{Duality between $\ft{S}^{2}\mP^{4}$ and the Double Quintic Symmetroid}
%\title{Hilbert scheme of conics in $\rG(n-1,n+1)$ and
%quadrics of rank $\leq 4$ in $\mP^n$}
%\title{Duality for double symmetric loci of rank at most four}
%\title{On the derived categories of symmetric determinanta loci}
%%
%% \title{Towards Homological Projective duality of 
%%  symmetric determinantal loci}
%%
%%   the shorter the better.
%%
\title{Towards Homological Projective duality 
for\\ ${\ft S}^2 \mP^3$ and ${\ft S}^2 \mP^4$}
\author{Shinobu Hosono and Hiromichi Takagi}
\begin{abstract}
We provide homological foundations to establish conjectural
homological projective dualities between
{1)} ${\ft S}^2 \mP^3$ and the double cover of the projective $9$-space
branched along the symmetric determinantal quartic,
and {2)} ${\ft S}^2 \mP^4$
and the double cover of the symmetric determinantal quintic in $\mP^{14}$
branched along the symmetric determinantal locus of rank at most $3$. 
\end{abstract}
\maketitle
\vspace{0.5cm}
%\tableofcontents
\markboth{Hosono and Takagi}
{HPD for ${\ft S}^2 \mP^3$ and ${\ft S}^2 \mP^4$}
\section{Introduction}

Throughout this paper,
we work over $\mC$, the complex number field.
We fix a vector space $V$ of dimension $n+1$
and denote by $V^*$ 
the dual vector space to $V$.

Let us consider the projective space $\mP({\ft S}^2 V^*)$ which 
we identify with the space of quadrics in  $\mP(V)$. In a separate 
paper \cite{Geom}, we have considered the locus 
$\nS_r\subset \mP({\ft S}^2 V^*)$ 
which represents quadrics in $\mP(V)$ of rank 
at most $r$. $\nS_r$ is a determinantal variety which is 
defined by $(r+1)\times (r+1)$ minors of the generic 
$(n+1)\times (n+1)$ symmetric matrix, and we have called it  
{\it symmetric determinantal locus of rank at most $r$}. 
As studied in [ibid.], when $r$ is even, $\nS_r$ has a double cover 
$\nT_r$ branched along $\nS_{r-1}$. $\nT_r$ is called {\it 
double symmetric determinantal locus of rank at most $r$}.  
These definitions apply in the same way using the dual projective spaces, 
i.e., $\nS_r^*$ and its double cover $\nT_r^*$ (when $r$ is 
even) are defined using the dual projective space $\mP({\ft S}^2 V)$.

In [ibid.], we have studied algebro-geometric 
properties of $\nS_r$ and $\nT_r$ in detail motivated by the so-called 
{\it homological projective duality} (HPD) due to Kuznetsov \cite{HPD1}.
HPD is a powerful framework to describe the derived category of a 
projective variety with its dual variety. Several interesting examples 
such as Pfaffian varieties (i.e., determinantal loci of 
anti-symmetric matrices) \cite{HPD2} and the second Veronese  
variety $\nS^*_1$  \cite{Quad} as well as the linear duality in general 
\cite[\S 8]{HPD1} have been studied.

The purpose of this paper is to lay homological foundations to
establish the HPDs for $\nS^*_2$ with $n=3,4$.
Indeed this paper is an extended version of the second part of \cite{Arxiv}
(the paper \cite{Geom} contains generalizations of the first part of [ibid.]). 
{It is useful to 
note that $\nS^*_2$ may be identified with 
${\ft S}^2 \mP(V)$ in a similar way to the relation of $\nS^*_1$ and 
the second Veronese variety $v_2(\mP(V))$.}
These conjectual HPDs are special cases
of two different types of plausible HPDs, which 
have naturally arisen from our 
algebro-geometric study on $\nS_r$ and $\nT_r$ in \cite[\S 3.5, 3.6]{Geom}.

The first one is the duality 
between $\nS^*_{n+2-r}$ and $\nT_{r}$ for each even $r$.
We suspect that their suitable non-commutative resolutions are 
HPD to each other with respect to certain (dual)
Lefschetz collections
because there exist orthogonal linear sections of
$\nS^*_{n+2-r}$ and $\nT_{r}$ 
such that they are Calabi-Yau varieties of the same dimension
(see \cite[Prop.~3.2 and 3.6]{Geom}).
In this paper, we consider the case where $n=r=4$, namely, we study
$\nS^*_2$ and $\nT_4$ for $n=4$.
In this case, 
there exist orthogonal linear sections 
{$X$ of $\nS^*_2$ and $Y$ 
of $\nT_4$}
such that they are {\it smooth} Calabi-Yau threefolds.
$X$ is so called a {\it Calabi-Yau threefold of Reye congruence},
and we call $Y$ a {\it double quintic symmetroid}.
In \cite{Geom}, we constructed certain resolutions 
$\widetilde{\nS}^*_2$ and $\widetilde{\nT}_4$
of $\nS^*_2$ and $\nT_4$, respectively.
In this paper, we construct (dual) Lefschetz collections
in the derived categories of the resolutions
$\widetilde{\nS}^*_2$ and $\widetilde{\nT}_4$ 
(Corollaries \ref{thm:Gvan1n=4} and \ref{thm:Gvan}). 
We remark that the (dual) Lefschetz collections
{have been} originally read off from a locally free resolution of
certain ideal sheaf on $\widetilde{\nT}_4\times \widetilde{\nS}^*_2$
(see \cite{HoTa3}).
Based on these (dual) Lefschetz collections,
we {have shown} in \cite{HoTa3} that $X$ and $Y$ are derived equivalent
(see also \cite{HoTa1, HoTa2}).
Moreover, we show that the dual Lefschetz collection in $\sD^b(\widetilde{\nS}^*_2)$ gives a dual Lefschetz decomposition of a categorical resolution of  
$\sD^b({\nS}^*_2)$
\footnote{Similar results have been obtained in \cite{Ren} using the 
category of matrix factorizations and the variation of GIT 
method \cite{BDFIK,DS,Halp}.}
 defined by Kuznetsov (Theorem \ref{thm:dualLef}).
These should be strong evidences for HPD
between $\nS^*_2$ and $\nT_4$.

The second plausible duality is between  
$\nS^*_{n+1-\frac r2}$ and $\nT_r$ for each even $r$. This duality may be 
observed in the resolution
$\widetilde{\nS}^*_{n+1-\frac r2}$ of $\nS^*_{n+1-\frac r2}$ 
and the fiber space $\nU_r$ of $\nT_r$ which have been constructed by 
using certain projective bundles over the Grassmannian 
$\rG(n+1-\frac{r}{2},V)$ \cite{Geom}.
It turns out that $\widetilde{\nS}^*_{n+1-\frac r2}$ and $\nU_r$ 
are given as certain incident varieties  
in $\mP({\ft S}^2 V)\times \rG(n+1-\frac r2,V)$ and 
$\mP({\ft S}^2 V^*)\times \rG(n+1-\frac r2,V)$, respectively, and 
are orthogonal to each other with respect to the dual pairing between $\ft{S}^2V$ and 
$\ft{S}^2V^*$. {We will see that} 
$\widetilde{\nS}^*_{n+1-\frac r2}$ and $\nU_r$ 
are precisely in the setting of 
the linear duality established by Kuzunetsov \cite[\S 8]{HPD1} and hence 
HPD to each other. This duality between 
$\widetilde{\nS}^*_{n+1-\frac r2}$ and 
$\nU_r$ indicates certain relationship between the derived categories of 
${\nS}^*_{n+1-\frac r2}$ and $\nT_r$. In this paper, to provide {a} supporting 
evidence, we consider the case of $n=3$ and $r=4$, i.e., we study $\nS^*_2$ and $\nT_4$ for $n=3$. In this case, we have 
an {\it Enriques surface of Reye congruence} 
and an {\it Artin-Mumford double solid} as 
an orthogonal linear sections of
$\nS^*_2$ and $\nT_4$, respectively.
We construct (dual) Lefschetz collections
in the derived categories of the resolutions $\widetilde{\nS}_2^*$ and
$\widetilde{\nT}_4$ 
(Corollaries \ref{thm:Gvan1} and \ref{thm:Gvann=3}). 
Based on these (dual) Lefschetz collections,
we {will} show in \cite{ReyeEnr} that 
there exists a close relationship between
the derived categories of the two linear sections.
Moreover, we  show that the dual Lefschetz collection 
in $\sD^b(\widetilde{\nS}^*_2)$ gives a dual Lefschetz 
decomposition of a categorical resolution of  
$\sD^b({\nS}^*_2)$ {$^{1)}$} defined by Kuznetsov (Theorem \ref{thm:dualLef}).
These should be strong evidences for HPD between $\nS^*_2$ and $\nT_4$.

\begin{ackn}
This paper is supported in part by Grant-in
Aid Scientific Research (S 24224001, B 23340010 S.H.) and Grant-in
Aid for Young Scientists (B 20740005, H.T.).  They thank Nicolas
Addington and Sergey Galkin for useful communications. {They also thank 
J\/orgen Vold Rennemo for letting us know about his Ph.D. thesis.}
\end{ackn}

\section{\bf{Basic results}}
\subsection{Borel-Weil-Bott Theorem}
We frequently use the following Borel-Weil-Bott Theorem.

For a locally free sheaf $\mathcal{E}$ of rank $r$ on a variety
and a nonincreasing sequence $\beta=(\beta_{1},\beta_{2},\dots,\beta_{r})$
of integers, we denote by $\ft{\Sigma}^{\beta}\mathcal{E}$ the associated
locally free sheaf with the Schur functor $\ft{\Sigma}^{\beta}$.

\begin{thm}\label{thm:Bott} Let $\pi\colon\mathrm{G}(r,\sA)\to X$
be a Grassmann bundle for a locally free sheaf $\sA$ on a variety
$X$ of rank $n$ and $0\to\sS\to\sA\to\sQ\to0$ the universal exact
sequence. For $\beta:=(\alpha_{1},\dots,\alpha_{r})\in\mZ^{r}$ $(\alpha_{1}\geq\dots\geq\alpha_{r})$
and $\gamma:=(\alpha_{r+1},\dots,\alpha_{n})\in\mZ^{n-r}$ $(\alpha_{r+1}\geq\dots\geq\alpha_{n})$,
we set $\alpha:=(\beta,\gamma)$ and $\sV(\alpha):=\ft{\Sigma}^{\beta}\sS^{*}\otimes\ft{\Sigma}^{\gamma}\sQ^{*}$.
Define $\rho:=(n,\dots,1)$ and for an element $\sigma$ of the $n$-th
symmetric group $\mathfrak{S}_{n}$, we set $\sigma\cdot(\alpha):=\sigma(\alpha+\rho)-\rho$.
Then the followings hold\,$:$

\begin{myitem2}

\item[\rm{(1)}] If $\sigma(\alpha+\rho)$ contains two equal integers,
then $R^{i}\pi_{*}\sV(\alpha)=0$ for any $i\geq0$. 

\item[\rm{(2)}]If there exists an element $\sigma\in\mathfrak{S}_{n}$
such that $\sigma(\alpha+\rho)$ is strictly decreasing, then $R^{i}\pi_{*}\sV(\alpha)=0$
for any $i\geq0$ except $R^{l(\sigma)}\pi_{*}\sV(\alpha)=\ft{\Sigma}^{\sigma\cdot(\alpha)}\sA^{*}$,
where $l(\sigma)$ represents the length of $\sigma\in\mathfrak{S}_{n}$.

\end{myitem2}

\end{thm}

\begin{proof} See \cite{Bo}, \cite{D}, or \cite[Cor.~(4.19)]{W}.
\end{proof}

\subsection{Basic definitions for triangulated categories}
We recall some basic definitions from the theory of triangulated categories
(cf.~\cite{Bondal,BO}).
\begin{defn}
\label{defn:excepcollect} An object $\sE$ in a triangulated category
$\sD$ is called \textit{an exceptional object} if $\Hom(\sE,\sE)\simeq\mC$
and $\Hom^{\bullet}(\sE,\sE)=0$ for $\bullet\not=0$. 
\end{defn}
\par
\begin{defn}
A triangulated subcategory $\sD'$ of $\sD$ is called \textit{admissible}
if there are right and left adjoint functors for the inclusion functor
$i_{*}:\sD'\to\sD$. 
\end{defn}
\par
\begin{defn}
\label{def:semi-orth} A sequence $\sD_{1},\dots,\sD_{m}$ of\textcolor{black}{{}
admissible} triangulated subcategories in a triangulated category
$\sD$ is called a \textit{semiorthogonal collection} if $\Hom_{\sD}(\sD_{i},\sD_{j})=0$
for any $i>j$. Moreover, if $\sD_{1},\dots,\sD_{m}$ generates $\sD$,
then it is called a \textit{semiorthogonal decomposition}.

A semiorthogonal collection of exceptional objects $\sE_{1},\dots,\sE_{n}$
is called \textit{an exceptional collection}. Moreover, if $\Hom^{\bullet}(\sE_{i},\sE_{j})=0$
holds for any $i,j$ and $\bullet\not=0$, then it is called \textit{a
strongly exceptional collection}. 
\end{defn}
Hereafter, in this article, we restrict our attention to the cases
of the derived categories of bounded complexes of coherent sheaves
on a variety. In such cases, a special type of semiorthogonal collection
plays an important role (cf.~\cite{HPD1,HPD2}).
\begin{defn}
For a variety $X$, a \textit{Lefschetz collection} of $\sD^{b}(X)$
is a semiorthogonal collection of the following form: \[
\sD_{0},\sD_{1}(1),\dots,\sD_{m-1}(m-1),\]
 where $0\subset\sD_{m-1}\subset\sD_{m-2}\subset\dots\subset\sD_{0}\subset\sD^{b}(X)$
and $(k)$ means the twist by $L^{\otimes k}$ with a fixed invertible
sheaf $L$. Moreover, if $\sD_{0},\sD_{1}(1),...,\sD_{m-1}(m-1)$
generate $\sD^{b}(X)$, then it is called a \textit{Lefschetz decomposition}.

Similarly, a \textit{dual Lefschetz collection} of $\sD^{b}(X)$ is
a semiorthogonal collection of the following form: \[
\sD_{m-1}(-(m-1)),\sD_{m-2}(-(m-2)),\dots,\sD_{0},\]
 where it holds that $0\subset\sD_{m-1}\subset\sD_{m-2}\subset\dots\subset\sD_{0}\subset\sD^{b}(X)$.
Moreover, if $\sD_{m-1}(-(m-1)),\sD_{m-2}(-(m-2)),\dots,\sD_{0}$
generate $\sD^{b}(X)$, then it is called a \textit{dual Lefschetz
decomposition}. 
\end{defn}
\;

\;

\section{{\bf Dual Lefschetz collection in $\sD^b(\hchow)$}}
\label{app:C}

\subsection{Symmetric determinantal loci $\nS^*_2$ and its Springer resolution $\widetilde{\nS}^*_2$}
\label{subsection:Sym}
We recall that $\nS^*_2\subset \mP({\ft S}^2 V)$ is the locus of quadrics in $\mP(V^*)$ of rank at most two. 
The Springer resolution of $\nS^*_2$ as in \cite[(2.1)]{Geom} is 
$\widetilde{\nS}_2^*:=\mP({\ft S}^2 \sF)$,
where
$\sF$ is the universal subbundle on $\rG(2,V)$.
As in \cite[Subsect.~3.2]{Geom},
we may identify $\nS^*_2$ with the Chow variety of length two $0$-cycles
in $\mP(V)$.
With this identification,
we may interpret \cite[Prop.~2.1 (2)]{Geom} in this situation as follow:
\[
\widetilde{\nS}_2^*:=\{([V_2],[\eta])\mid \Supp\, \eta\subset \mP(V_2)\}
\subset \rG(2,V)\times \nS^*_2,
\]
where we consider $\eta$ is a length two $0$-cycle in $\mP(V)$.  
From this description of $\widetilde{\nS}_2^*$,
it is easy to see that $\widetilde{\nS}_2^*$ is the Hilbert scheme of 
length two $0$-dimensional subschemes of $\mP(V)$ and
the Springer resolution $\widetilde{\nS}_2^*\to \nS^*_2$
is the Hilbert-Chow morphism. 
Note that $\widetilde{\nS}_2^*\to \nS^*_2$ is isomorphic over 
$[\eta]\in \nS^*_2$ if the support of $\eta$ consists of two distinct points since
then $\eta$ determines the line $\mP(V_2)$ 
such that $\eta\subset \mP(V_2)$ uniquely.
If the support of $\eta$ consists of one point, say $x$,
then the fiber over $[\eta]$ parameterizes lines through $x$.
Therefore the exceptional locus $E_f$ of 
$\widetilde{\nS}_2^*\to \nS^*_2$ is a prime divisor isomorphic to
$\mP(\sF)$. In particular, 
$E_f$ has a $\mP^{n-1}$-bundle structure over the rank one locus $\nS^*_1\simeq \mP(V)$.
\subsection{Homological properties of certain locally free sheaves on $\hchow$}
For brevity of notation, we set
\[
\chow:=\nS^*_2,\, \hchow:=\widetilde{\nS}^*_2,\, f\colon \hchow\to \chow,\,
g\colon \hchow\to \rG(2,V).
\]
We also denote by $H_{\hchow}$ and $L_{\hchow}$ the pull-back of $\sO_{\mP({\ft S}^2 V)}(1)$ and $\sO_{\rG(2,V)}(1)$. 
For brevity of notation, we often omit
the subscripts $\empty_\hchow$
from $H_{\hchow}$ and $L_{\hchow}$, 
and $g^*$ for the pull-back of coherent sheaves to $\hchow$ of $\rG(2,V)$.
 
We consider the Euler sequence
\[
0\to\sO_{\hchow}(-H)\to \ft{S}^{2}\sF\to T_{\hchow/G(2,V)}(-H)\to0
\]
associated to $\mP({\ft S}^2 \sF)\to \rG(2,V)$.
Twisting this by $L$ we obtain 
\begin{equation}
0\to\sO_{\hchow}(-H+L)\to \ft{S}^{2}\sF(L)\to T_{\hchow/G(2,V)}(-H+L)\to0.
\label{eq:F1a}\end{equation}

\begin{thm}
\label{thm:digGvanX}
Suppose $n=3,4$.
\begin{enumerate}[$(1)$]
\item
The ordered sequence $\sO_{\hchow}(-H)$,
$\sO_{\hchow}(-L)$, 
$\sF$, or
$T_{\hchow/\rG(2,V)}(-H+L)$
is semi-orthogonal.
\item
Let
${\sA}$ or ${\sB}$ be one of
the locally free sheaves 
$\sO_{\hchow}(-H)$,
$\sO_{\hchow}(-L)$, 
$\sF$, or
$T_{\hchow/\rG(2,V)}(-H+L)$ on $\hchow$.
Then 
\[
H^{\bullet}({\sA}^*\otimes {\sB} (-t))=0 \ \text{for}\ 
\begin{cases}
t=1 : n=3,\\
t=2,3: n=3 \ \text{and} \ {\sB}\not= T_{\hchow/\rG(2,V)}(-H+L),\\
1\leq t \leq 4: n=4.
\end{cases}
\]
\end{enumerate}
\end{thm}

\vspace{0.8cm}
 We prepare the following lemma for our proof of the theorem in case $n=4$. The
lemma follows from \cite[Prop.4.8]{Lef}, but for the reader's convenience
we present a proof. 
\begin{lem}
\label{cla:key1} 
Suppose $n=4$. For $\sC=\sA^*\otimes \sB$, 
it holds that \[
H^{\bullet}(\hchow,\sC(-t))\simeq H^{8-\bullet}(\hchow,\sC^{*}(t-5))\text{ for any }t.\]
 \end{lem}

\begin{proof} 
By \cite[\S 2.1]{Geom},
we have $K_{\hchow}=-5H+E_f$
where $E_f$ is the $f$-exceptional divisor.
By the Serre duality,
we have $H^{\bullet}(\hchow,\sC(-t))\simeq H^{8-\bullet}(\hchow,\sC^{*}((t-5)H+E_f))$
for any $t$. By the exact sequence \[
0\to\sC^{*}((t-5))\to\sC^{*}((t-5)H+E_f)\to\sC^{*}((t-5)H+E_f)|_{E_f}\to0,\]
 we have only to show that $H^{8-\bullet}(E_f,\sC^{*}((t-5)H+E_f)|_{E_f})=0$.
As we see in Subsection \ref{subsection:Sym}, 
the image of $E_f$ by $f$ is $\nS^*_1$
and $E_f\to\nS^*_1\simeq\mP^{4}$ is a $\mP^{3}$-bundle. Therefore it
suffices to show the vanishing of cohomology groups of the restriction
of $\sC^{*}((t-5)H+E_f)|_{E_f}$ to a fiber $\Gamma$
of $E_f\to\nS^*_2$. By [ibid.], $\sO_{\hchow}(E_f)|_{\Gamma}\simeq\sO_{\mP^{3}}(-2)$
and $\sO_{\hchow}(H)|_{\Gamma}\simeq\sO_{\mP^{3}}$. 
As we see in Subsection \ref{subsection:Sym}, 
$E_f$ parameterizes
pairs of a point $x\in\mP(V)$ and a line $l$ through $x$. Therefore
a fiber $\Gamma\cong\mathbb{P}^{3}$ parameterizes lines through a
fixed point (i.e., $V_{1}\subset V_{2}\subset V$ for a fixed point
$x=[V_{1}])$. This implies that $g^{*}\sF^{*}|_{\Gamma}\simeq\sO_{\mP^{3}}\oplus\sO_{\mP^{3}}(1)$.
Restricting the natural injection $\sO_{\hchow}(-H)\to g^{*}\ft{S}^{2}\sF$
to $\Gamma$, we have an injection \[
\sO_{\mP^{3}}\to\sO_{\mP^{3}}\oplus\sO_{\mP^{3}}(-1)\oplus\sO_{\mP^{3}}(-2).\]
 Therefore, by the Euler sequence, we have $T_{\hchow/\rG(2,V)}(-H+L)|_{\Gamma}\simeq\sO_{\mP^{3}}\oplus\sO_{\mP^{3}}(-1)$.
Consequently, $\sC^{*}((t-5)H+E_f)|_{\Gamma}$ is a direct
sum of $\sO_{\mP^{3}}(-1)$, $\sO_{\mP^{3}}(-2)$ and $\sO_{\mP^{3}}(-3)$
for any $\sC$ and $t$, hence all of its cohomology
groups vanish. \end{proof}

\vspace{0.5cm}
 \noindent\textit{\textcolor{black}{Proof of Theorem $\ref{thm:digGvanX}$.}}
In any case, we can calculate the cohomology groups in a similar way.
Thus we only give computations only for $n=4$ and 
$\sA=\sF,\, \sB=T_{\hchow/\rG(2,V)}(-H+L)$.
Note that we may assume that $t=0,1,2$ by Lemma \ref{cla:key1},
which simplifies the computations considerably.
Twisted (\ref{eq:F1a}) with $\sF^*(-tH)$, we obtain
\begin{align}
\label{eq:F1attw}
0\to\sF^*(-(t+1)H+L)\to 
\sF^*\otimes \ft{S}^{2}\sF(-tH+L)\to\\ \sF^*\otimes T_{\hchow/G(2,V)}(-(t+1)H+L)\to0.\nonumber
\end{align}
We compute the cohomology groups 
of $\sF^*(-(t+1)H+L)$ and 
$\sF^*\otimes \ft{S}^{2}\sF(-tH+L)$.
We see that 
$H^{\bullet}(\sF^*(-(t+1)H+L))$ is $0$ for $t=0,1$ since
$\hchow\to \rG(2,V)$ is a $\mP^2$-bundle.
To compute
$H^{\bullet}(\sF^*(-3H+L))$, we take its Serre dual
$H^{8-\bullet}(\hchow,\sF(-3L))\simeq
H^{8-\bullet}(\rG(2,V),\sF^*(-4))$, which vanish
by Theorem \ref{thm:Bott}.
We also see that 
$H^{\bullet}(\sF^*\otimes \ft{S}^{2}\sF(-tH+L))$ is $0$ for $t=1,2$ since
$\hchow\to \rG(2,V)$ is a $\mP^2$-bundle, and
\[
H^{\bullet}(\sF^*\otimes \ft{S}^{2}\sF(L))\simeq
H^{\bullet}(\rG(2,V),\sF^*\otimes \{\ft{S}^{2}\sF\} (1))
\simeq H^{\bullet}(\rG(2,V),{\ft \Sigma}^{2,-1}\sF^*\oplus \sF^*),
\]
which vanish except for $\bullet=0$, and 
\[
H^0(\rG(2,V),{\ft \Sigma}^{2,-1}\sF^*\oplus \sF^*)\simeq
H^0(\rG(2,V),\sF^*)\simeq V^*
\]
by Theorem \ref{thm:Bott}.
Therefore, by (\ref{eq:F1attw}),
we have 
$H^{\bullet}(\sF^*\otimes T_{\hchow/G(2,V)}(-(t+1)H+L))$
is $0$ except for $\bullet=0$ and $t=0$, and
$H^0(\sF^*\otimes T_{\hchow/G(2,V)}(-H+L))\simeq V^*$.

\hfill $\square$

\subsection{Dual Lefschetz collection in $\sD^{b}(\hchow)$}

\label{subsection:LefX}

It is straightforward to obtain the following result from Theorem 
\ref{thm:digGvanX}.

 \begin{cor}
\label{thm:Gvan1}
Suppose $n=3$.
Let $\Lambda:=\left\{ 3,2,1_{a},1_{b}\right\}$
be an ordered set $(\Lambda,\prec)$. 
Define \[
(\sF_{\alpha})_{\alpha\in\Lambda}:=
(\sF_{3},\sF_{2},\sF_{1a},\sF_{1b})=
(\Omega^1_{\hchow/\rG(2,V)}(-L+H),\, \sF^*,\, \sO_{\hchow}(L),\,\sO_{\hchow}(H))
\]
be an ordered collection of sheaves.
We define the following triangulated subcategories of $\sD^b(\hchow):$
\begin{eqnarray*}
\sD^0_{\chow}&=&\sD^1_{\chow}:=\langle 
\mathcal{F}^*_{1b}, \mathcal{F}^*_{1a},
\mathcal{F}^*_{2}, \mathcal{F}^*_3\rangle,\\
\sD^2_{\chow}&=&\sD^3_{\chow}:=\langle 
\mathcal{F}^*_{1b}, \mathcal{F}^*_{1a},
\mathcal{F}^*_{2}\rangle.
\end{eqnarray*}
Then 
\begin{equation*}
\label{eq:dualLef}
\sD_{\chow}^3(-3),\sD_{\chow}^2(-2), \sD_{\chow}^1(-1),
\sD_{\chow}^0
\end{equation*}
is a dual Lefschetz collection,
where $(-t)$ represents 
the twist by the sheaf $\sO_{\hchow}(-tH)$.  
\end{cor}

\begin{rem}
We can obtain the following results by a similar method to show
Theorem \ref{thm:digGvanX}.
%Since we do not use these in this paper though interesting, we omit the proof.

\begin{enumerate}[(1)]
\item
We see that $\Ext^4(\mathcal{F}^*_3, \mathcal{F}^*_3(-2))\simeq \mC$.
This is the reason for the elimination of $\mathcal{F}^*_3$ from 
$\sD_{\chow}^0$.
\item
$(\mathcal{F}_{\alpha})_{\alpha\in\Lambda}$
is a strongly exceptional collection in $\sD^b(\hchow)$.
\item 
$\Hom$'s of the sheaves in the above collection are 
given by the following diagram:
\begin{equation*}
\begin{matrix}
\begin{xy}
(5,-12)*+{\circ}="c3",
(20,-12)*+{\circ}="c2",
(35,0)*+{\circ}="c1a",
(35,-24)*+{\circ}="c1b",
(0,-12)*{\mathcal{F}_3},
(20,-15)*{\mathcal{F}_2},
(38,3)*{\mathcal{F}_{1a}},
(38,-27)*{\mathcal{F}_{1b}},
(18,1)*{{\ft S}^2V^*},
(18,-25)*{{\ft S}^2 V^*},
(30,-8)*{V^*},
(15,-10)*{V^*},
(30,-16)*{V^*},
\ar "c3";"c2"
\ar "c2";"c1a"
\ar "c2";"c1b"
\ar @/^1pc/@{->} "c3";"c1a"
\ar @/_1pc/@{->} "c3";"c1b"
\end{xy}
\end{matrix}
%\label{eq:quiver1}
\end{equation*}
\end{enumerate}
\end{rem}

\vspace{1cm}

In case $n=4$,
the following dual Lefschetz collection is suitable for our purpose (see \cite{HoTa3}).

\begin{cor}
\label{thm:Gvan1n=4} 
Suppose $n=4$.
Let $\Lambda:=\left\{ 3,2,1_{a},1_{b}\right\} $
be an ordered set $(\Lambda,\prec)$. 
Define an ordered collection of sheaves on $\hchow;$
\[
\begin{aligned}(\mathcal{F}_{\alpha})_{\alpha\in\Lambda} & :=(\mathcal{F}_{3},\mathcal{F}_{2},\mathcal{F}_{1a},\mathcal{F}_{1b})\\
 & =(\,\sO_{\hchow},\; \sF^{*},T_{\hchow/G(2,V)}(-H+2L),\sO_{\hchow}(L)\,).\end{aligned}
\]
Set $\sD_{\hchow}:=\langle\mathcal{F}_{1{b}}^{*},\mathcal{F}_{1{a}}^{*},\mathcal{F}_{2}^{*},\mathcal{F}_{3}^{*}\rangle\subset\sD^{b}(\hchow).$
 Then \[
\sD_{\hchow}(-4),\dots,\sD_{\hchow}(-1),\sD_{\hchow}\]
 is a dual Lefschetz collection, where $(-t)$ represents the twist
by the sheaf $\sO_{\hchow}(-tH)$. 
\end{cor}

\begin{proof}
By taking the dual of the sheaves of Theorem \ref{thm:digGvanX} (1),
we have the following dual Lefschetz collection:
\begin{align*}
\Omega_{\hchow/\rG(2,V)}(-3H-L),\,
\sF^*(-4H),\,
\sO_{\hchow}(-4H+L),\,  
\sO_{\hchow}(-3H),  
\\
\dots\dots\\
\Omega_{\hchow/\rG(2,V)}(H-L),\,
\sF^*,\,
\sO_{\hchow}(L),\,  
\sO_{\hchow}(H).
\end{align*}
Let $\sC$ be one of the sheaves in this collection 
except $\sO_{\hchow}(H)$.
Then by Lemma \ref{cla:key1}, 
it holds that \[
\Hom^{\bullet}(\sO_{\hchow}(H),\sC)\simeq \Hom^{8-\bullet}(\sC,\sO_{\hchow}(-4H)).\]
Therefore we obtain 
the following dual Lefschetz collection.
\begin{align*}
\sO_{\hchow}(-4H),\,
\Omega_{\hchow/\rG(2,V)}(-3H-L),\,
\sF^*(-4H),\,
\sO_{\hchow}(-4H+L),\,   
\\
\dots\dots\\
\sO_{\hchow},\,\Omega_{\hchow/\rG(2,V)}(H-L),\,
\sF^*,\,
\sO_{\hchow}(L).
\end{align*}
Tensoring this with
$\sO_{\hchow}(-L)$,
we obtain the desired  
dual Lefschetz collection.
%$\sD_{\hchow}(-4),\dots,\sD_{\hchow}(-1),\sD_{\hchow}$.
\end{proof}
\begin{rem}
We may also obtain the following results by a similar method to show
Theorem \ref{thm:digGvanX} (see \cite{Arxiv}).

\begin{enumerate}[(1)]
\item
$(\mathcal{F}_{\alpha})_{\alpha\in\Lambda}$ is a strongly exceptional
collection of $\sD^{b}(\hchow)$.
\item 
$\Hom$'s of the sheaves in the above collection are 
given by the following diagram$:$ 
\begin{equation}
\begin{matrix} 
\begin{xy} 
(5,-12)*+{\circ}="c3", 
(20,-12)*+{\circ}="c2", 
(35,0)*+{\circ}="c1a", 
(35,-24)*+{\circ}="c1b", 
(0,-12)*{\mathcal{F}_3}, 
(20,-15)*{\mathcal{F}_2}, 
(38,3)*{\mathcal{F}_{1_a}}, 
(38,-27)*{\mathcal{F}_{1_b}}, 
(18,1)*{{\ft S}^2V^*}, 
(18,-25)*{\wedge^2V^*}, 
(30,-8)*{V^*}, 
(15,-10)*{V^*}, 
(30,-16)*{V^*}, 
\ar "c3";"c2" 
\ar "c2";"c1a" 
\ar "c2";"c1b" 
\ar @/^1pc/@{->} "c3";"c1a" 
\ar @/_1pc/@{->} "c3";"c1b" 
\end{xy} 
\end{matrix}  
\end{equation}
\end{enumerate}
\end{rem}

\subsection{Categorical and noncommutative resolution of $\sD^b(\chow)$}

Since $\hchow\to\chow$ is a resolution of rational singularities
whose exceptional locus is a prime divisor $E_{f}$, we have
a triangulated subcategory $\check{\sD}\subset\sD^{b}(\hchow)$ called
a categorical resolution of $\sD^{b}(\chow)$ for every dual Lefschetz
decomposition of $\sD^{b}(E_{f})$ \cite[Theorem 1]{Lef}. There is
a natural dual Lefschetz decomposition of $\sD(E_{f})$ for the $\mP^{n-1}$-bundle $E_{f}\to\nS^*_1=v_{2}(\mP(V))$ \cite{Sa}:
\begin{enumerate}[(1)]
\item
($n$ is even)
$$
\sD^b(E_f) = \langle \sC_{\frac n2-1}(-n+2),\dots,\sC_1(-2),\sC_0\rangle,
$$
where $\sC_0 = \sC_1 = \dots = \sC_{\frac n2-1} = \langle 
(f|_{E_f})^*{\sD^b(\mP(V))}(-1),
(f|_{E_f})^*{\sD^b(\mP(V))}\rangle$,
where $(-k)$ is the twist by $\sO_{\hchow}(-kL)|_{E_f}$.
\item
($n$ is odd)
$$
\sD^b(E_f) = \langle \sC_{\frac {n-1}{2}}(-n+1),\dots,\sC_1(-2),\sC_0\rangle,
$$
where $\sC_0 = \sC_1 = \dots = \sC_{\frac {n-3}{2}} = \langle 
(f|_{E_f})^*{\sD^b(\mP(V))}(-1),
(f|_{E_f})^*{\sD^b(\mP(V))}\rangle$,
and
$\sC_{\frac {n-1}{2}} = \langle 
(f|_{E_f})^*{\sD^b(\mP(V))}\rangle$.
\end{enumerate}
Let $\check{\sD}$ be the triangulated subcategory of $\sD^b(\hchow)$
which consists of objects $F$ such that $\iota^*F\in \sC_0$,
where $\iota$ is the natural closed embedding $E_f\to \hchow$.
By \cite{Lef},
$\sD^b(\hchow)$ has the following semi-orthogonal decomposition:
$$
\sD^b(\hchow)=
\begin{cases} 
\langle \iota_*\sC_{\frac n2-1}(-n+2),\dots,\iota_*\sC_1(-2), \check{\sD}\rangle:\text{$n$ is even}.\\
\langle \iota_*\sC_{\frac {n-1}{2}}(-n+1),\dots,\iota_*\sC_1(-2), \check{\sD}\rangle:\text{$n$ is odd}.
\end{cases}
$$
Recall that $\chow$ is Gorenstein if and only if $n$ is even \cite[\S 2.1]{Geom}.
When $n$ is even, $\check{\sD}$ is strongly crepant.
Indeed, in this case, the conditions of
[ibid., Prop.~4.7] holds:
\begin{itemize}
\item $\sC_0 = \sC_1 = \dots = \sC_{\frac n2-1}$(the decomposition is
called {\it{rectangular}}), and
\item 
the discrepancy of $f$ is $\frac{n-2}{2}$ (\cite[\S 2.1]{Geom}),
which is equal to the length of the decomposition of
$\sD^b(E_f)$.
\end{itemize}

The categorical resolution is also related to
the noncommutative resolution by Van den Bergh (\cite[Theorem 2]{Lef}).
It is easy to see that
\[
\sR:=f_*\sE\!\text{\it{nd}}\, (\sO_{\hchow}\oplus \sO_{\hchow}(-L))
\] 
satisfies the assumptions of [ibid.];
to check $\sC_0$ is generated by $\iota^* \sE$
is obvious, and 
to check $\sE$ is tilting follows from the standard relative vanishing theorem.
Thus 
\[
\check{\sD}\simeq \sD(\chow,\sR).
\]
\subsection{Dual Lefschetz decomposition of the categorical resolution for $n=2,3,4$}
\label{sub:decomp}
\begin{thm}
\label{thm:dualLef}
For $n=2, 3, 4$,
we have the following dual Lefschetz decomposition of
$\check{\sD}:$ 
\[
\check{\sD}\simeq 
\langle \sA_n(-n),\dots, \sA_1(-1),\sA_{0} \rangle,
\]
where
\begin{eqnarray*}
\sA_0&=&\sA_1=\sA_2=
\langle \sO_{\hchow}(-H), \sO_{\hchow}(-L),
\sF\rangle \ \text{for $n=2$},\\
\sA_0&=&\sA_1=
\langle \sO_{\hchow}(-H),\, \sO_{\hchow}(-L),\, \sF,\, T_{\hchow/\mathrm{G}(2,V)}(-H+L)\rangle,\ \text{and}\\
\sA_2&=&\sA_3=
\langle \sO_{\hchow}(-H),\, \sO_{\hchow}(-L),\, \sF\rangle
\ \text{for $n=3$},\\
\sA_0&=&\cdots=\sA_4=
\langle \sO_{\hchow}(-H),\, \sO_{\hchow}(-L),\, \sF,\, T_{\hchow/\mathrm{G}(2,V)}(-H+L)\rangle \ \text{for $n=4$}.\\
\end{eqnarray*}
\end{thm}

We set 
\[
\text{$\hchow_n:=\mP({\ft S}^2 \sF),\, \chow_n:={\ft S}^2 \mP(V)$
for $V$ with $\dim V=n+1$.}
\]

\begin{proof}[{\bf Proof of Theorem $\ref{thm:dualLef}$ in case $n=2$}]~

Note that $\check{\sD}$ is equivalent to $\sD^b(\hchow_2)$
since $\hchow_2\to \chow_2$ is crepant.
Since $\hchow_2$ is a $\mP^2$-bundle over the projective plane $\mathrm{G}(2,V)\simeq \mP(V^*)$,
the derived category $\sD^b(\hchow_2)$ has the following standard semi-orthogonal decomposition by Beilinson's and Orlov's results:
\begin{equation*}
%\label{eqn:hchow2}
%\begin{scriptsize}
\begin{matrix}
\sD^b(\hchow_2)=\langle &\sO_{\hchow_3}(-2H-L),& \sF(-2H),& \sO_{\hchow_3}(-2H) &\\
& \sO_{\hchow_3}(-H-L),& \sF(-H),& \sO_{\hchow_3}(-H),\\
& \sO_{\hchow_3}(-L),& \sF,& \sO_{\hchow_3} &\rangle.
\end{matrix}
%\end{scriptsize}
\end{equation*}
Since $K_{\hchow_2}=-3H$,
we obtain by mutating $\sO_{\hchow_2}$ to the left;
\begin{equation}
\label{eqn:hchow2}
%\begin{scriptsize}
\begin{matrix}
\sD^b(\hchow_2)=\langle &\sO_{\hchow_3}(-3H),& \sO_{\hchow_3}(-2H-L),& \sF(-2H),& &\\
     & \sO_{\hchow_3}(-2H),& \sO_{\hchow_3}(-H-L),& \sF(-H),& \\
       & \sO_{\hchow_3}(-H),& \sO_{\hchow_3}(-L),& \sF& \rangle,
\end{matrix}
%\end{scriptsize}
\end{equation}
which is nothing but the desired result for $n=2$.
\end{proof}

The rest of this subsection is occupied with our proof of 
Theorem $\ref{thm:dualLef}$ in case $n=3,4$.

We have already shown
that $\sA_n(-n),\dots,\sA_1(-1), \sA_{0}$ is semi-orthogonal
in Theorem \ref{thm:digGvanX}.
Besides, they are contained in the left orthogonal to
\[
\iota_*\sC_1(-2):=
\begin{cases}
\iota_*(f|_{E_f})^*{\sD^b(\mP^3)}(-2): n=3,\\
\langle 
(f|_{E_f})^*{\sD^b(\mP^4)}(-3),
(f|_{E_f})^*{\sD^b(\mP^4)}(-2)\rangle: n=4
\end{cases}
\]
since
the restrictions of
\[
\sO_{\hchow_n}(-H),\, \sO_{\hchow_n}(-L),\, \sF,\, T_{\hchow_n/\mathrm{G}(2,V)}(-H+L)
\]
to a fiber\,$\simeq \mP^{n-1}$ of $E_f\to \mP(V)$
are direct sums of $\sO_{\mP^{n-1}}$ and $\sO_{\mP^{n-1}}(-1)$
(cf.~the proof of Lemma \ref{cla:key1}).

We set
\[
\begin{cases}
\sA:=\sA_0=\sA_1,\, \sA':=\sA_2=\sA_3: n=3,\\
\sA:=\sA_0=\cdots=\sA_4: n=4.
\end{cases}
\]
We define the following triangulated subcategory $\sC\subset \sD^b(\hchow_n)$:
\[
\sC:=
\begin{cases}
\empty^{\perp}\langle \iota_*\sC_1(-2), \sA'(-3), \sA'(-2), \sA(-1), \sA \rangle: n=3,\\
\empty^{\perp}\langle \iota_*\sC_1(-2), \sA(-4), \sA(-3), \sA(-2), \sA(-1), \sA \rangle: n=4.
\end{cases}
\]
Then, by \cite{Bondal},
we have the following semiorthogonal decomposition of $\sD^b(\hchow_n)$:
\begin{equation}
\label{eqn:start}
\begin{scriptsize}
\begin{matrix}
\sD^b(\hchow_3)= \langle  \iota_*\sC_1(-2), &\sO_{\hchow_3}(-4H),& \sO_{\hchow_3}(-3H-L),& \sF(-3H),& &\\
       & \cellcolor[gray]{0.85} \sO_{\hchow_3}(-3H),& \cellcolor[gray]{0.85}\sO_{\hchow_3}(-2H-L),& \cellcolor[gray]{0.85}\sF(-2H),& &\\
     & \cellcolor[gray]{0.85}\sO_{\hchow_3}(-2H),& \cellcolor[gray]{0.85}\sO_{\hchow_3}(-H-L),& \cellcolor[gray]{0.85}\sF(-H),& T_{\hchow_3/\rG(2,V)}(-2H+L)& \\
       & \cellcolor[gray]{0.85}\sO_{\hchow_3}(-H),& \cellcolor[gray]{0.85}\sO_{\hchow_3}(-L),& \cellcolor[gray]{0.85}\sF,& T_{\hchow_3/\rG(2,V)}(-H+L),&\sC\rangle.
\end{matrix}
\end{scriptsize}
\end{equation}
Note that the restriction of the gray part of (\ref{eqn:start}) 
corresponds to the r.h.s. of (\ref{eqn:hchow2}).
\begin{equation}
\label{eqn:start2}
\begin{scriptsize}
\begin{matrix}
\sD^b(\hchow_4)= \langle  \iota_*\sC_1(-2), &\sO_{\hchow_3}(-5H),& \sO_{\hchow_3}(-4H-L),& \sF(-4H),& T_{\hchow_3/\rG(2,V)}(-5H+L) &\\
       & \cellcolor[gray]{0.85} \sO_{\hchow_3}(-4H),& \cellcolor[gray]{0.85}\sO_{\hchow_3}(-3H-L),& \cellcolor[gray]{0.85}\sF(-3H),& T_{\hchow_3/\rG(2,V)}(-4H+L)&\\
       & \cellcolor[gray]{0.85} \sO_{\hchow_3}(-3H),& \cellcolor[gray]{0.85}\sO_{\hchow_3}(-2H-L),& \cellcolor[gray]{0.85}\sF(-2H),&T_{\hchow_3/\rG(2,V)}(-3H+L) &\\

     & \cellcolor[gray]{0.85}\sO_{\hchow_3}(-2H),& \cellcolor[gray]{0.85}\sO_{\hchow_3}(-H-L),& \cellcolor[gray]{0.85}\sF(-H),& \cellcolor[gray]{0.85}T_{\hchow_3/\rG(2,V)}(-2H+L)& \\
       & \cellcolor[gray]{0.85}\sO_{\hchow_3}(-H),& \cellcolor[gray]{0.85}\sO_{\hchow_3}(-L),& \cellcolor[gray]{0.85}\sF,& \cellcolor[gray]{0.85}T_{\hchow_3/\rG(2,V)}(-H+L),&\sC\rangle.
\end{matrix}
\end{scriptsize}
\end{equation}
Note that the restriction of the gray part of (\ref{eqn:start2}) 
corresponds to a part of the r.h.s. of (\ref{eqn:start}).

It remains to show the fullness of the collection
$\iota_*\sC_1(-2), \sA_n(-n),\dots, \sA_1(-1),\sA_{0}$, equivalently,
$\sC=0$.
Following the inductive argument of Kuznetsov in the proof of \cite[Thm.~4.1]{KuzGrIsotropic}, we reduce the proof for $\hchow_3$ to the fullness of the collection
for $\hchow_2$, and then
the proof for $\hchow_4$ to the fullness of the collection
for $\hchow_3$.
For this,
we take an $n$-dimensional vector subspace $V'\subset V$.
Then we simply denote 
$\mP({\ft S}^2 \sF)$ over $\rG(2,V')$ by $\hchow_{n-1}$
as a subvariety of $\hchow_n$. 
Let $j\colon \hchow_{n-1}\hookrightarrow \hchow_n$ 
be the natural closed immersion.
It is well-known (see \cite[Lem.~4.4]{KuzGrIsotropic})
that the Koszul resolution of 
$j_*\sO_{\hchow_{n-1}}$ is the following form: 
\begin{equation}
\label{eq:red}
0\to \sO_{\hchow_n}(-L)\to
\sF\to \sO_{\hchow_{n-1}}\to j_*\sO_{\hchow_{n-1}}\to 0.
\end{equation}
Then, by a similar result to \cite[Lem.~4.5]{KuzGrIsotropic},
we have only to show $j^*\sE=0$ for any object $\sE$ of $\sC$
since we may choose $V'$ freely and
we may assume the fullness of the collection for $\hchow_{n-1}$.
For this, arguing as in \cite[p.165]{KuzGrIsotropic},
we have only to show 
the following claim by (\ref{eq:red}):
\begin{lem}
\label{lem:PQ}
The following sheaves of the form $\sP\otimes \sQ$
are contained in $\sC^{\perp}$\,$:$
\begin{itemize}
\item $(n=3)$
$\sP$ is one of the sheaves in the gray part of 
$(\ref{eqn:start})$, namely,
\[
\sP=\sO_{\hchow_3}(-iH)\, (1\leq i\leq 3),\,
\sO_{\hchow_3}(-iH-L)\, (0\leq i\leq 2),\,
\sF(-iH)\, (0\leq i\leq 2),
\]
$\sQ$ is one of the sheaves in the exact sequence
$(\ref{eq:red})$ except $j_*\sO_{\hchow_2}$, namely,
\[
\sQ= 
\sO_{\hchow_3}(-L),\,
\sF,\, \sO_{\hchow_3}.
\]
\item
$(n=4)$
$\sP$ is one of the sheaves in the gray part of 
$(\ref{eqn:start2})$, namely,
\begin{align*}
\sP=\sO_{\hchow_4}(-iH)\, (1\leq i\leq 4),\,
\sO_{\hchow_4}(-iH-L)\, (0\leq i\leq 3),\,
\sF(-iH)\, (0\leq i\leq 3),\\
T_{\hchow_3/\rG(2,V)}(-iH+L)
\, (i=1,2),
\end{align*}
$\sQ$ is one of the sheaves in the exact sequence
$(\ref{eq:red})$ except $j_*\sO_{\hchow_3}$, namely,
\[
\sQ= 
\sO_{\hchow_4}(-L),\,
\sF,\, \sO_{\hchow_4}.
\]
\end{itemize}
\end{lem}

To show Lemma \ref{lem:PQ},
we prepare the following three results:

\begin{lem}
\label{lem:exactseq}
We denote by $H_{\mP(\sF)}$ the tautological divisor of $E_f\simeq \mP(\sF)$.
For any $k\in \mZ$, there exist the following exact sequences\,$:$
\begin{align}
0\to \sO_{\hchow}(-2H-kL)\to \sO_{\hchow}(-(2+k)L)\to
\sO_{E_f}(-(2+k)L)\to 0,\label{exa3}\\
0\to \sF(-H-kL)\to \sF(-(k+1)L)\to
\sO_{E_f}(H_{\mP(\sF)}-(k+2)L)\to 0,\label{exa1}\\
0\to T_{\hchow/\mathrm{G}(2,V)}(-2H-kL)\to\qquad \qquad \qquad\qquad \qquad\qquad\qquad\label{exanew}\\
T_{\hchow/\mathrm{G}(2,V)}(-H-(k+1)L)\to\sO_{E_f}(2H_{\mP(\sF)}-(k+3)L) \to 0.\nonumber\\
0\to \sO_{\hchow}(-H)\to {\ft S}^2 \sF\to
T_{\hchow/\mathrm{G}(2,V)}(-H)\to 0,\label{exa0}\\
0\to T_{\hchow/\mathrm{G}(2,V)}(-H)
\to {\ft S}^2 \sF(H-L)\to
\sO_{\hchow}(2H-3L)\to 0,\label{exa1'}
\end{align}
\end{lem}
\begin{proof}
Noting that
\begin{equation}
\label{eq:Ef}
E_f\sim 2(H-L) \quad \text{(\cite[\S 2.1]{Geom})},
\end{equation}
(\ref{exa3}) is obtained from the standard exact sequence
$0\to \sO_{\hchow}(-E_f)\to \sO_{\hchow}\to
\sO_{E_f}\to 0$ by tensoring $\sO_{\hchow}(-(k+2)L)$.
(\ref{exa0}) is nothing but the Euler sequence.
(\ref{exa1'}) is obtained by dualizing and twisting
(\ref{exa0}).

We will construct the exact sequence (\ref{exa1}). Twisting
$\sO_{\hchow_n}(kL)$, it suffices to show the existence of
the exact sequence
\begin{equation}
\label{eq:-H-L}
0\to \sF(-H)\to \sF(-L)\to
\sO_{E_f}(H_{\mP(\sF)}-2L)\to 0,
\end{equation}
where we note that $\sO_{\hchow_n}(H)|_{E_f}\simeq \sO_{E_f}(2H_{\mP(\sF)})$.
The construction below is nothing but a relativization of 
the Kapranov's construction of the spinor sheave $\sO_{\mP^1}(1)$ on a plane conic.
By the Littlewood-Richardson rule,
we have ${\ft S}^2 \sF\otimes \sF\simeq {\ft S}^3 \sF\oplus \sF(-1)$ on $\rG(2,V)$.
Therefore we have a map 
$p \colon {\ft S}^2 \sF\otimes \sF\to \sF(-1)$ by the projection to the second factor. 
Let $[V_2]\in \rG(2,V)$ be any point.
It is convenient to identify ${\ft S}^k V_2$ with the spaces of 
binary $k$-forms with variables $x$, $y$.
Then the map $p$
coincides up to constant with the $\SL(V_2)$-equivariant map
$\delta\colon {\ft S}^2 V_2\otimes V_2\to V_2$
such that
\[
\delta(f,g)=\frac{\partial f}{\partial x} \frac{\partial g}{\partial y}-
\frac{\partial f}{\partial y} \frac{\partial g}{\partial x}
\]
for $f\in {\ft S}^2 V_2$ and $g\in V_2$. 
Computing explicitly,
we see that
the composite
\[
{\ft S}^2 V_2\otimes{\ft S}^2 V_2\otimes V_2\overset{1\otimes \delta}{\to} 
{\ft S}^2 V_2\otimes V_2\overset{\delta}{\to} V_2
\]
satisfies
$(f,f,g)\mapsto (b^2-4ac)g$, where $f=ax^2+bxy+cy^2$.
Since ${\ft S}^2 ({\ft S}^2 V_2)\simeq {\ft S}^4 V_2\oplus \mC$ as an $\SL(V_2)$-module, 
the map ${\ft S}^2 V_2\otimes{\ft S}^2 V_2\to \mC$ satisfying
$(f,f)\mapsto b^2-4ac$ corresponding to
the projection ${\ft S}^2 ({\ft S}^2 V_2)\to \mC$
to the second factor.
Therefore
the composite 
\[
{\ft S}^2 \sF \otimes{\ft S}^2 \sF \otimes \sF \overset{1\otimes p}{\to} 
{\ft S}^2 \sF\otimes \sF(-1)\overset{p}{\to} \sF(-2)
\]
coincides with the map
${\ft S}^2 \sF \otimes{\ft S}^2 \sF \otimes \sF \to
\sO_{\rG(2,V)}(-2) \otimes \sF$
induced from
the projection 
\begin{equation}
\label{eq:projS2}
{\ft S}^2 \sF \otimes{\ft S}^2 \sF\to
{\ft S}^2 ({\ft S}^2 \sF)\simeq {\ft S}^4 \sF\oplus \sO_{\rG(2,V)}(-2)
\to
\sO_{\rG(2,V)}(-2).
\end{equation}
Since the map 
${\ft S}^2 ({\ft S}^2 \sF) \to
\sO_{\rG(2,V)}(-2)$ corresponds to
the conic fibration $E_f\to \rG(2,V)$,
we see that $p$ is the Clifford multiplication associated to 
$E_f\to \rG(2,V)$.
Therefore, by the construction of the spinor sheaf $\sO_{\mP^1}(1)$ on a plane conic, we see that
the map on $\hchow$
\[
\sO_{\hchow}(-H)\otimes \sF\hookrightarrow
{\ft S}^2 \sF \otimes \sF\overset{p}{\to}\sF(-L)
\]
is injective and the cokernel $\sL$ is an invertible sheave 
on $E_f$ of the form $\sO_{E_f}(H_{\mP(\sF)}+aL)$ with some $a\in \mZ$.
Since the induced map $\sF(-L)\to \sL$ is $\SL(V)$-equivariant,
we see that $a=-2$. Therefore we have obtained the desired exact sequence
$(\ref{eq:-H-L})$.

The construction of (\ref{exanew}) is similar to that of (\ref{exa1}),
so we only give a sketch below.
Twisting
$\sO_{\hchow_n}(kL)$, it suffices to show the existence of
the exact sequence
\begin{equation}
\label{eq:-H-L'}
0\to T_{\hchow/\mathrm{G}(2,V)}(-2H)\to
T_{\hchow/\mathrm{G}(2,V)}(-H-L)\to\sO_{E_f}(2H_{\mP(\sF)}-3L) \to 0.
\end{equation}
By the Littlewood-Richardson rule,
we have ${\ft S}^2 \sF\otimes {\ft S}^2 \sF\simeq {\ft S}^4 \sF\oplus {\ft S}^2 \sF (-1)\oplus \sO_{\rG(2,V)}(-2)$ on $\rG(2,V)$.
Therefore we have a map 
$q \colon {\ft S}^2 \sF\otimes {\ft S}^2 \sF\to {\ft S}^2 \sF(-1)$ by the projection to the middle factor. 
We see that this map is locally defined as the map $\delta$ above.
Now we consider the composite of 
the maps on $\hchow$
\[
\sO_{\hchow}(-H)\otimes {\ft S}^2 \sF\hookrightarrow
{\ft S}^2 \sF \otimes {\ft S}^2 \sF\overset{q}{\to}{\ft S}^2 \sF(-L).
\]
Then, by local computations, we see that 
the restriction of this map to 
$\sO_{\hchow}(-H)\otimes \sO_{\hchow}(-H)$ is zero.
Therefore, noting ${\ft S}^2 \sF/\sO_{\hchow}(-H)\simeq T_{\hchow/\rG(2,V)}(-H)$,
we have the following map:
\[
\overline{q}\colon \sO_{\hchow}(-H)\otimes T_{\hchow/\rG(2,V)}(-H) 
\to {\ft S}^2 \sF(-L)\to T_{\hchow/\rG(2,V)}(-H-L).
\]
We consider
the composite
\begin{align}
\label{align:-HT}
\sO_{\hchow}(-H) \otimes \sO_{\hchow}(-H) \otimes T_{\hchow/\rG(2,V)}(-H) \overset{1\otimes \overline{q}}{\longrightarrow} \qquad\qquad\qquad \\
\sO_{\hchow}(-H) \otimes T_{\hchow/\rG(2,V)}(-H-L)\overset{\overline{q}}{\to} 
T_{\hchow/\rG(2,V)}(-H-2L).\nonumber
\end{align}
Note that, by (\ref{eq:Ef}), 
we obtain a map
\begin{equation}
\label{eq:-2H-2L}
\sO_{\hchow}(-2H)\to \sO_{\hchow}(-2L)
\end{equation}
twisting the natural map 
$\sO_{\hchow}(-E_f)\hookrightarrow \sO_{\hchow}$ with $\sO_{\hchow}(-2L)$.
As in the case of (\ref{exanew}), we see by local computations that
(\ref{align:-HT}) coincides with the map
$(\ref{eq:-2H-2L})\otimes T_{\hchow/\rG(2,V)}(-H)$.
Therefore, $\overline{q}$ is a Clifford multiplication.
Again, by local computations, we see that $\overline{q}$ is
injective and the cokernel $\sM$ is an invertible sheave 
on $E_f$ of the form $\sO_{E_f}(2H_{\mP(\sF)}+bL)$ with some $b\in \mZ$.
Since the induced map $T_{\hchow/\rG(2,V)}(-H-L)\to \sM$ is $\SL(V)$-equivariant,
we see that $b=-3$. Therefore we have obtained the desired exact sequence
$(\ref{eq:-H-L'})$.

\end{proof}

\begin{cor}
\label{cor:key}
For $n=3,4$, the following sheaves are contained in $\sC^{\perp}$\,$:$
\begin{enumerate}[$(1)$]
\item
$\sO_{\hchow_n}(-iH-2L)$ 
for $n=3,4$ and $-1\leq i\leq n-1$.
\item
$\sO_{\hchow_4}(-iH-3L)$ 
for $n=4$ and $-2\leq i\leq 2$.
\item
$\sF(-iH-L)$ 
for $n=3,4$ and $-1\leq i\leq n-1$.
\item
$\sF(-iH-2L)$ 
for $n=4$ and $-2\leq i\leq 2$.
\item
$T_{\hchow_n/\rG(2,V)}(-iH)$ 
for $n=3$ and $i=0,1$, and $n=4$ and $0\leq i\leq 4$.
\item
$T_{\hchow_4/\rG(2,V)}(-iH-L)$ 
for $n=4$ and $-1\leq i\leq 3$.

\end{enumerate}
\end{cor}

\begin{proof}
All the assertions can be proved in a similar way by using 
the exact sequences in Lemma \ref{lem:exactseq}.
Thus it should suffice to prove some of them.

To show (3),
we consider $(\ref{exa1}) \otimes \sO_{\hchow_3}(-iH)$ with $k=0$:
\[
0\to \sF(-(i+1)H)\to \sF(-iH-L)\to
\sO_{E_f}((-2i+1)H_{\mP(\sF)}-2L)\to 0.
\]
Since
$\sO_{E_f}((-2i+1)H_{\mP(\sF)}-2L)\in \iota_*\sC(-2)\subset \sC^{\perp}$,
and $\sF(-(i+1)H)\in \sC^{\perp}$ for 
$0\leq i+1\leq n$ by (\ref{eqn:start}),
we have 
$\sF(-iH-L)\in \sC^{\perp}$ for $-1\leq i\leq n-1$.

To show (4),
we consider $(\ref{exa1}) \otimes \sO_{\hchow_3}(-iH)$ with $k=1$:
\[
0\to \sF(-(i+1)H-L)\to \sF(-iH-2L)\to
\sO_{E_f}((-2i+1)H_{\mP(\sF)}-3L)\to 0.
\]
Since
$\sO_{E_f}((-2i+1)H_{\mP(\sF)}-3L)\in \iota_*\sC(-2)\subset \sC^{\perp}$ for 
$n=4$,
and $\sF(-(i+1)H-L)\in \sC^{\perp}$ for 
$-1\leq i+1\leq 3$ by (3) as we have proved,
we have
$\sF(-iH-2L)\in \sC^{\perp}$
for
$-2\leq i\leq 2$.
\end{proof}

\begin{lem}
\label{lem:long}
\begin{enumerate}[$(1)$]
\item 
For $n=3$,
there exist the following exact sequences\,$:$
\begin{align}
0\to 
{\ft S}^2 \sF\to 
V\otimes \sF\to \wedge^2 V\otimes \sO_{\hchow_3}\to
\sO_{\hchow_3}(L) \to 0,\label{ex1}\qquad\qquad\qquad\\
\ \ 0\to
\sO_{\hchow_3}(-3L)
\to
\wedge^2 V^* \otimes \sO_{\hchow_3}(-2L) 
\to \label{ex2}
V^*\otimes \sF(-L)\to
{\ft S}^2 \sF \to 0.
%0\to \sF^*\to V\otimes \sO_{\hchow_3}(L)\to
%V^* \otimes \sO_{\hchow_3}(2L)\to
%\sF^*(2L) \to 0.
\end{align}
\item For $n=4$,
there exist the following exact sequences\,$:$
\begin{align}
0\to
{\ft S}^2 \sF
\to
V \otimes \sF \to 
\wedge^2 V\otimes \sO_{\hchow_4}\to\qquad\qquad\label{ex4}\\
V^*\otimes \sO_{\hchow_4}(L)
\to
\sF^*(L)\to 0\nonumber\\
0\to \sO_{\hchow_4}(-4L)\to 
\wedge^3 V^*\otimes \sO_{\hchow_4}(-3L)\to\qquad\qquad\label{ex3}\\
\wedge^2 V^*\otimes \sF(-2L)
\to
V^*\otimes {\ft S}^2 \sF(-L)\to
{\ft S}^3 \sF\to 0.\nonumber
\end{align}
\end{enumerate}
\end{lem}
\begin{proof}
See \cite[the proof of Lem.~4.3]{KuzGrIsotropic}.
\end{proof}

\begin{proof}[{\bf Proof of Lemma \ref{lem:PQ}}]
First we assume that $n=3$.

\vspace{5pt}

\noindent {\bf Case $\sQ=\sO_{\hchow_3}$.} The assertion holds since
$\sP\in\sC^{\perp}$.

\vspace{5pt}

\noindent {\bf Case $\sQ=\sO_{\hchow_3}(-L)$.}
If $\sP=\sO_{\hchow_3}(-iH)\, (1\leq i\leq 3)$, then 
$\sP(-L)=\sO_{\hchow_3}(-iH-L)\in \sC^{\perp}$ by (\ref{eqn:start}).
If $\sP=\sO_{\hchow_3}(-iH-L)$ $(0\leq i\leq 2)$, then
$\sP(-L)=\sO_{\hchow_3}(-iH-2L)\in \sC^{\perp}$ by
Corollary \ref{cor:key} (1).
If $\sP=\sF(-iH)$
$(0\leq i\leq 2)$, then
$\sP(-L)=\sF(-iH-L)\in \sC^{\perp}$ by
Corollary \ref{cor:key} (3).

\vspace{5pt}

\noindent {\bf Case $\sQ=\sF$.}
If $\sP=\sO_{\hchow_3}(-iH)$ ($1\leq i\leq 3$), then
$\sF\otimes \sP=\sF(-iH)\in \sC^{\perp}$ by (\ref{eqn:start}).
If $\sP=\sO_{\chow_3}(-iH-L)$ $(0\leq i\leq 2)$, then 
$\sF\otimes \sP\simeq \sF(-iH-L)\in \sC^{\perp}$ by
Corollary \ref{cor:key} (3).
Finally we assume that $\sP=\sF(-iH)$ $(0\leq i\leq 2)$.
We note that $\sF\otimes \sF(-iH)\simeq \sO_{\hchow_3}(-iH-L)\oplus
{\ft S}^2 \sF(-iH)$.
Since $\sO_{\hchow_3}(-iH-L)\in\sC^{\perp}$ by (\ref{eqn:start}),
it remains to show ${\ft S}^2 \sF(-iH)\in \sC^{\perp}$.
For ${\ft S}^2 \sF$, it suffices to show
that
$\sO_{\hchow_3}(-H),\, T_{\hchow/\mathrm{G}(2,V)}(-H)\in \sC^{\perp}$ 
by $(\ref{exa0})$.
By (\ref{eqn:start}), $\sO_{\hchow_3}(-H)\in \sC^{\perp}$.
Moreover, by 
Corollary \ref{cor:key} (5), we see that 
$T_{\hchow/\mathrm{G}(2,V)}(-H)\in \sC^{\perp}$.
The argument for ${\ft S}^2 \sF(-iH)$ ($i=1,2$) is slightly
involved. We use also Lemma \ref{lem:long}.
By considering $(\ref{ex1}) \otimes \sO_{\hchow_3}(-iH)$, 
we are reduced to show  
$\sF(-iH),\,
\sO_{\hchow_3}(-iH),\,
\sO_{\hchow_3}(-iH+L)\in \sC^{\perp}$.
By (\ref{eqn:start}),
$\sF(-iH), \sO_{\hchow_3}(-iH)\in \sC^{\perp}$.
For $\sO_{\hchow_3}(-iH+L)$,
we consider $(\ref{exa0})\otimes \sO_{\hchow_3}(-(i-1)H+L)$.
Then
we are reduce to show that ${\ft S}^2 \sF(-(i-1)H+L)\in \sC^{\perp}$
since $T_{\hchow/\mathrm{G}(2,V)}(-iH+L)\in \sC^{\perp}$ by (\ref{eqn:start}).
For  ${\ft S}^2 \sF(-(i-1)H+L)$, we consider
$(\ref{ex2})\otimes \sO_{\hchow_3}(-(i-1)H+L)$.
Then we are reduced to show that
$\sO_{\hchow_3}(-(i-1)H-2L),\, \sO_{\hchow_3}(-(i-1)H-L),\,
\sF(-(i-1)H)\in \sC^{\perp}$.
The latter two sheaves is contained in $\sC^{\perp}$ by 
(\ref{eqn:start}), and  
$\sO_{\hchow_3}(-(i-1)H-2L)\in \sC^{\perp}$
by Corollary \ref{cor:key} (1).

\vspace{5pt}

Next we consider the case $n=4$.
\vspace{5pt}

\noindent {\bf Case $\sQ=\sO_{\hchow_4}$.} The assertion holds since
$\sP\in\sC^{\perp}$.

\vspace{5pt}

\noindent {\bf Case $\sQ=\sO_{\hchow_4}(-L)$.}
If $\sP=\sO_{\hchow_4}(-iH)\, (1\leq i\leq 4)$, then 
$\sP(-L)=\sO_{\hchow_4}(-iH-L)\in \sC^{\perp}$ by (\ref{eqn:start2}).
If $\sP=\sO_{\hchow_4}(-iH-L)$ $(0\leq i\leq 3)$, then
$\sP(-L)=\sO_{\hchow_4}(-iH-2L)\in \sC^{\perp}$ by
Corollary \ref{cor:key} (1).
If $\sP=\sF(-iH)$
$(0\leq i\leq 3)$, then
$\sP(-L)=\sF(-iH-L)\in \sC^{\perp}$ by
Corollary \ref{cor:key} (3).
If $\sP=T_{\hchow_3/\rG(2,V)}(-iH+L)$ $(i=1,2)$,
then $\sP(-L)=T_{\hchow_3/\rG(2,V)}(-iH)\in \sC^{\perp}$
by Corollary \ref{cor:key} (5).

\vspace{5pt}

\noindent {\bf Case $\sQ=\sF$.}
If $\sP=\sO_{\hchow_4}(-iH)$ ($1\leq i\leq 4$), then
$\sF\otimes \sP=\sF(-iH)\in \sC^{\perp}$ by (\ref{eqn:start2}).
If $\sP=\sO_{\chow_4}(-iH-L)$ $(0\leq i\leq 3)$, then 
$\sF\otimes \sP\simeq \sF(-iH-L)\in \sC^{\perp}$ by
Corollary \ref{cor:key} (3).
For $\sP=\sF(-iH)$ $(0\leq i\leq 3)$,
we note that $\sF\otimes \sF(-iH)\simeq \sO_{\hchow_4}(-iH-L)\oplus
{\ft S}^2 \sF(-iH)$.
By (\ref{eqn:start}), $\sO_{\hchow_4}(-iH-L)\in\sC^{\perp}$. 
For ${\ft S}^2 \sF(-iH)\in \sC^{\perp}$,
we consider $(\ref{exa0})\otimes \sO_{\hchow_4}(-iH)$.
Then it suffices to show 
that
$\sO_{\hchow_4}(-(i+1)H), T_{\hchow/\mathrm{G}(2,V)}(-(i+1)H)\in \sC^{\perp}$.
$\sO_{\hchow_4}(-(i+1)H)\in \sC^{\perp}$ by (\ref{eqn:start2}).
We have $T_{\hchow/\mathrm{G}(2,V)}(-(i+1)H)\in \sC^{\perp}$ by 
Corollary \ref{cor:key} (5). 
Finally we assume that $\sP=T_{\hchow_4/\rG(2,V)}(-iH+L)$ $(i=1,2)$.
The argument below is slightly involved. 
Considering $(\ref{exa1'})\otimes \sF(-(i-1)H+L)$, we are reduced to show
that ${\ft S}^2 \sF\otimes \sF(-(i-2)H),\, \sF(-(i-3)H-2L)\in \sC^{\perp}$.
We have 
$\sF(-(i-3)H-2L)\in \sC^{\perp}$ by Corollary \ref{cor:key} (4). 
As for ${\ft S}^2 \sF\otimes \sF(-(i-2)H)$, we note that the decomposition
${\ft S}^2 \sF\otimes \sF(-(i-2)H)\simeq {\ft S}^3 \sF(-(i-2)H)\oplus
\sF(-L-(i-2)H)$.
We have $\sF(-L-(i-2)H)\in \sC^{\perp}$ by Corollary \ref{cor:key} (3). 
Now we show that ${\ft S}^3 \sF(-(i-2)H)\in \sC^{\perp}$.
Considering $(\ref{ex3})\otimes \sO_{\hchow_4}(-(i-2)H)$,
we are reduced to show
that
$\sO_{\hchow_4}(-(i-2)H-4L),
\sO_{\hchow_4}(-(i-2)H-3L),
\sF(-(i-2)H-2L),
{\ft S}^2 \sF(-(i-2)H-L)\in \sC^{\perp}$.
We have $\sO_{\hchow_4}(-(i-2)H-3L)$ and $\sF(-(i-2)H-2L)\in \sC^{\perp}$
by Corollary \ref{cor:key} (2) and (4) respectively. 
As for 
${\ft S}^2 \sF(-(i-2)H-L)$,
considering 
$(\ref{exa0})\otimes \sO_{\hchow_4}(-(i-2)H-L)$,
we are reduce to show 
$\sO_{\hchow_4}(-(i-1)H-L)\in \sC^{\perp}$, which follows from 
(\ref{eqn:start2}),
and $T_{\hchow_4/\rG(2,V)}(-(i-1)H-L)\in \sC^{\perp}$,
which follows from Corollary \ref{cor:key} (6).
To show that
$\sO_{\hchow_4}(-(i-2)H-4L)\in \sC^{\perp}$,
we consider $(\ref{exa1'})\otimes \sO_{\hchow_4}(-iH-L)$.
Then we are reduced to show that
$T_{\hchow_4/\rG(2,V)}(-(i+1)H-L)\in \sC^{\perp}$,
which follows from Corollary \ref{cor:key} (6), and
${\ft S}^2 \sF(-(i-1)H-2L)\in \sC^{\perp}$.
For the latter, we consider $(\ref{ex4})\otimes \sO_{\hchow_4}(-(i-1)H-2L)$.
Then we are reduced to show that
$\sF(-(i-1)H-2L), \sO_{\hchow_4}(-(i-1)H-2L),
\sO_{\hchow_4}(-(i-1)H-L), \sF(-(i-1)H)\in \sC^{\perp}$,
which follow from Corollary \ref{cor:key} (4) and (1), and 
(\ref{eqn:start2}), respectively.
\end{proof}
Now we have finished our proof of Theorem $\ref{thm:dualLef}$.
$\hfill\square$

\begin{rem}
We believe that 
a similar method work for any $n$ as in \cite{KuzGrIsotropic}
once we can find a suitable
candidate of the dual Lefschetz collection of maximal length.
\end{rem}

\section{{\bf Locally free sheaves $\tilde{\sS}_{L}$, $\tilde{\sQ}$, $\tilde{\frQ}$ on $\widetilde{\hcoY}$}\label{sec:SheavesDef}}
In this section, 
$n$ is any integer greater than or equal to $3$.
\subsection{Birational geometry of the double symmetric loci $\nT_4$}
\label{sub:BirGeom}
As in \cite[\S 3.4]{Geom},
we set
\[
\Hes:=\nS_4,\, \UU:=\widetilde{\nS}_4,\, \hcoY:=\nT_4,\, \Zpq:=\nU_4,
\]
where $\nT_4$ is the double cover of $\nS_4$ branched along $\nS_3$.
We quickly review the main result of \cite{Geom}, which describe
the 
birational geometry of $\hcoY$.

Let 
\[
\hcoY_3:=\mathrm{G}(3,\wedge^2 \mathfrak{Q})
\]
with the universal quotient bundle $\frQ$ of $\rG(n-3,V)$.
In case $n=3$, we consider $\rG(n-3,V)$ is a point and 
$\frQ$ is the vector space $V$.
We denote by $\Prt_{\rho}$ and $\Prt_{\sigma}$
the subvarieties of $\hcoY_3$ parameterizing $\rho$-planes
and $\sigma$-planes respectively (we refer for the definitions
of $\rho$-planes
and $\sigma$-planes to \cite[\S 4.1]{Geom}).
In [ibid., \S 4.5], we have seen
that $\Prt_{\rho}\simeq \rF(n-3, n-2;V)$ and
$\Prt_{\sigma}\simeq \rF(n-3, n;V)$,
where $\rF(a,b;V):=\{(\mC^a, \mC^b)\mid \mC^a\subset \mC^b\subset V\}$
is the flag variety.

 In [ibid.], we construct the following diagram:
\def\FigYsRed{\resizebox{10cm}{!}{\includegraphics{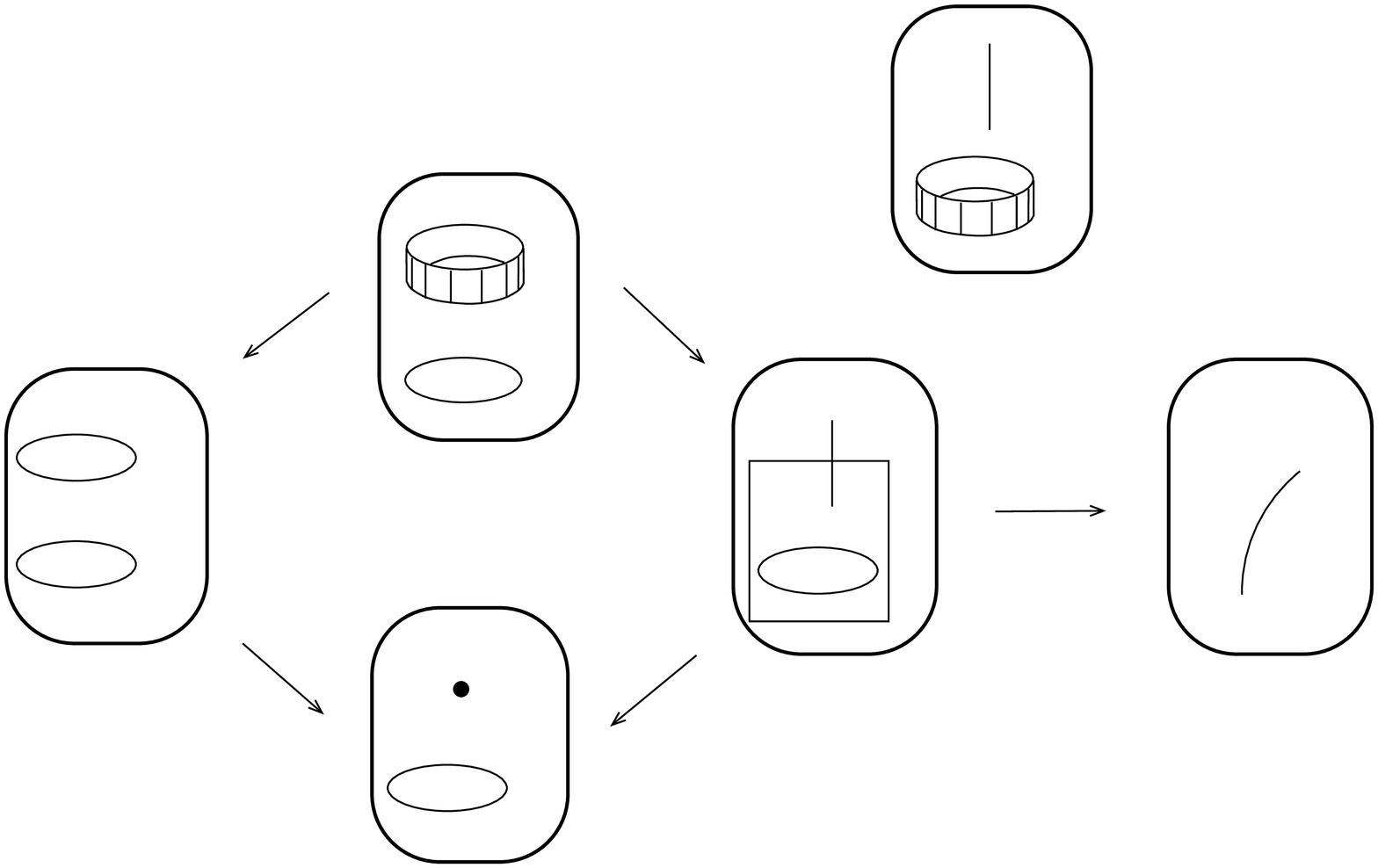}}} 
\def\FigYsDisplay{ 
\begin{xy} 
(0,6)*{\FigYsRed}, 
%%%%%%%%%%%%%%%%%% 
(-54,5.5)*{\hcoY_3}, 
(-38.5,3.5)*{\Prt_\rho}, 
(-38.5,-5)*{\Prt_\sigma},
(-28,12)*{\,_{\Lrho_{\hcoY_2}}}, 
%%%%%%%%%%%%%%%%%% 
{(-41.5,-12) \ar 
(-41.5,-20)^{\Lpi_{\hcoY_3}}_{\overset{\rG(3,6)\text{-}}{\text{bundle}}}}, 
(-41.5,-24)*{\rG(n-3,V)}, 
%%%%%%%%%%%%%%%%% 
{(-31,0) \ar @{-->} (-2,0)^{\text{(anti-)flip}}}, 
%%%%%%%%%%%%%%%%% 
(-25,24)*{\hcoY_2}, 
(-10, 16.5)*{F_\rho}, 
(-11,9)*{\Prt_\sigma}, 
(-2,11)*{\,_{\tLrho_{\hcoY_2}}}, 
%%%%%%%%%%%%%%%%% 
(10,15)*{\widetilde\hcoY}, 
(7,1.7)*{F_{\widetilde\hcoY}}, 
(13,7.5)*{G_\rho}, 
(15,-5)*{\Prt_\sigma}, 
(26,-3)*{\,_{\Lrho_{\widetilde\hcoY}}}, 
%%%%%%%%%%%%%%%%% 
(10,30)*{\hcoY_0}, 
(27,22)*{F_{\sigma}}, 
(25,30)*{G_\rho}, 
{(10,26) \ar (10,19)},
%%%%%%%%%%%%%%%%% 
(-16,-3.8)*{\overline{\hcoY}'}, 
(-28,-8.7)*{\,_{\Lrho_{\hcoY_3}}}, 
( -3,-9.1)*{\,_{\Lp_{\widetilde\hcoY}}}, 
(-16,-15)*{\overline{\Prt}_{\sigma}}, 
(-11.5,-20.5)*{\overline{\Prt}_\sigma}, 
%%%%%%%%%%%%%%%%%
(41.5,14)*{\hcoY}, 
(44.5,-3)*{G_\hcoY}, 
%%%%%%%%%%%%%%%%%%% 
%(-26.5,36)*{\hcoY_1}, 
%(-10.5,34)*{F_\rho}, 
%(-10.5,25)*{\Gamma_\sigma}, 
%%%%%%%%%%%%%%%%%% 
%(21.5,36)*{\hcoY_0}, 
%(13.5,34)*{\Gamma_\rho}, 
%(16.5,25)*{\Gamma_\sigma}, 
%%%%%%%%%%%%%%%%%% 
\end{xy} }

\[
\FigYsDisplay\]
where
\begin{itemize}
\item $\overline{\hcoY}'$ is the normalization
of the subvariety $\overline{\hcoY}$ of $\rG(3,\wedge^{n-1} V)$
parametrizing $3$-planes annihilated by at least $n-3$ linearly independent vectors in $V$ by the wedge product ([ibid., Prop.~4.8, 4.9]),
\item
$\hcoY_3\to \overline{\hcoY}'$
is a small contraction
contracting $\Prt_{\rho}$ to $\overline{\Prt}_{\rho}\simeq \rG(n-2,V)$
([ibid., Prop.~4.11])
with the isomorphic image $\overline{\Prt}_{\sigma}$ of ${\Prt}_{\sigma}$,
\item
$\hcoY_3\dashrightarrow \widetilde{\hcoY}$
is the $($anti-$)$\,flip for the small contraction  
$\hcoY_3\to \overline{\hcoY}'$ ([ibid., \S 4.4]),
\item
$\Lp_{\widetilde{\hcoY}}\colon\widetilde{\hcoY}\to \overline{\hcoY}'$
is a small contraction
contracting $G_{\rho}\simeq \mP({\ft S}^2 \sQ_{\rho}^*)$ to
$\overline{\Prt}_{\rho}\simeq\rG(n-2,V)$, where 
$\sQ_{\rho}$ is the universal quotient bundle on $\rG(n-2,V)$
([ibid., Prop.~4.15]),
\item
$\Lrho_{\hcoY_2}\colon \hcoY_2\to \hcoY_3$ is the blow-up
along the subvariety $\Prt_{\rho}$ ([ibid., \S 4.5 and \S 4.7]),
\item
$\widetilde{\Lrho}_{\hcoY_2}\colon \hcoY_2\to \widetilde{\hcoY}$ is the blow-up
along the subvariety $G_{\rho}$ of codimension $n-2$ ([ibid., Prop.~4.22 (1)]),
\item
$\Lrho\,_{\widetilde{\hcoY}}\colon \widetilde{\hcoY}\to \hcoY$ is an extremal divisorial contraction with exceptional divisor $F_{\widetilde{\hcoY}}$
([ibid., \S 4.6, Prop.~4.22 (2), \S 5]), 
\item
$\hcoY_0\to \widetilde{\hcoY}$ is the blow-up along 
the strict transform of $\Prt_{\sigma}$ ([ibid., \S 4.4, Rem.~4.23]).
\end{itemize}

\begin{rem}
In case $n=3$,
we consider $\hcoY_3=\overline{\hcoY}'=\rG(3,\wedge^2 V)$
and $\hcoY_2\simeq \widetilde{\hcoY}$.
\end{rem}

\vspace{5pt}

In the subsequent subsections, 
we introduce locally free sheaves $\tilde{\sS}_{L}$, $\tilde{\sQ}$,
$\tilde{\frQ}$ on $\widetilde{\hcoY}$, which will play central roles
in our construction of the Lefschetz collection in $\sD^b(\widetilde{\hcoY})$. 

In Sections \ref{sec:SheavesDef} and
\ref{section:Lefschetz-Y}, and Appendix \ref{sec:Appendix-D},
 we use the following convention for the invertible
sheaves:

\begin{myitem} 

\item[$\;\quad\quad L_{\Sigma}$:] the pull back on a variety $\Sigma$
of $\sO_{\rG(n-3,V)}(1)$ if there is a morphism $\Sigma\to\rG(n-3,V)$. 
In case $n=3$, we consider $L_{\Sigma}$ as the trivial sheaf
$\wedge^4 V\otimes \sO_{\Sigma}$.

\item[$\quad\quad M_{\Sigma}$:] the pull back on a variety $\Sigma$
of $\sO_{\nS_4}(1)$ if there is a morphism $\Sigma\to\nS_4$. \end{myitem}

We often omit the subscripts $\Sigma$ for $L_{\Sigma}$ and $M_{\Sigma}$
if no confusion is likely possible. 

\subsection{Locally free sheaves $\widetilde{\sS}_{L}$, $\widetilde{\sQ}$ on $\widetilde{\hcoY}$}

Consider the following universal sequence of the Grassmann bundle
$\hcoY_{3}=$ $\rG(3,\wedge^{2}\frQ)$ over $\rG(n-3,V)$ (cf. \cite[p.434]{Ful}):
\begin{equation}
0\to\sS\to\Lpi_{\hcoY_{3}}^{\;*}{\wedge^2\frQ}\to\sQ\to0,\label{eq:univ}\end{equation}
where $\sS$ is the relative universal subbundle of rank three and
$\sQ$ is the relative universal quotient bundle of rank three. Similarly,
we denote by $\bar{\sS}$ the universal subbundle of rank three of
the Grassmannian $\rG(3,\wedge^{n-1}V)$. Then we have
\begin{prop}
$\sS(-L_{\hcoY_{3}})$ is the pull-back of $\bar{\sS}$. 
\end{prop}
\begin{proof} We may write a point of $\hcoY_{3}$ by $y=([\bar{U}],[V_{n-3}])$
with $[\bar{U}]\in \rG(3,\wedge^2 (V/V_{n-3}))$, 
$[V_{n-3}]\in \rG(n-3,V)$. 
$y$ is mapped to $[U]=[\bar{U}\wedge \wedge^{n-3} V_{n-3}]\in\overline{\hcoY}$. Note
that $\wedge^{n-3} V_{n-3}=\Lpi_{\hcoY_{3}}^{\;*}\sO_{\mP(V)}(-1)\vert_{y}=-L_{\hcoY_{3}}\vert_{y}$.
Therefore $\sS(-L_{\hcoY_{3}})$ is the pull-back of $\bar{\sS}$. 

\end{proof}

Now we have the following proposition (and definition):
\begin{prop}
\label{def:SQT1} There exist locally free sheaves $\widetilde{\sS}_{L}$
and $\widetilde{\sQ}$ on $\widetilde{\hcoY}$ which satisfy \[
\Lrho_{\hcoY_{2}}^{\;*}\sS(-L_{\hcoY_{2}})=\tLrho_{\hcoY_{2}}^{\;*}\widetilde{\sS}_{L}\text{ and }\Lrho_{\hcoY_{2}}^{\;*}\sQ=\tLrho_{\hcoY_{2}}^{\;*}\widetilde{\sQ}.\]

\end{prop}
\begin{proof} We define $\widetilde{\sS}_{L}$ to be the pullback
of $\bar{\sS}$ to $\widetilde{\hcoY}$, then the first claim is immediate
by the commutativity of the morphisms in the diagram in Subsection \ref{sub:BirGeom}. To see the existence
of $\widetilde{\sQ}$, consider the universal sequence (\ref{eq:univ})
on $\hcoY_{3}$. Let $[V_{n-3},V_{n-2}]$ be a point on the exceptional
locus $\Prt_{\rho}=\rF(n-3,n-2,V)\to\overline{\Prt}_{\rho}=\rG(n-2,V)$ of the small resolution $\hcoY_{3}\to\overline{\hcoY}'$.
Since $\sS\vert_{[V_{n-3},V_{n-2}]}=(V/V_{n-2})\wedge(V_{n-2}/V_{n-3})$ by
\cite[Prop.~4.1]{Geom}, we have $0\to(V/V_{n-2})\wedge(V_{n-2}/V_{n-3})\to\wedge^{2}(V/V_{n-3})\to\sQ\vert_{[V_{n-3},V_{n-2}]}\to0$.
Hence we have $\sQ\vert_{[V_{n-3},V_{n-2}]}\simeq\wedge^{2}(V/V_{n-2})$,
which implies $\sQ\vert_{\gamma}\simeq\sO_{\gamma}^{\oplus3}$ for the fiber
$\gamma=\mP^{n-3}$ of $\rF(n-3,n-2,V)\to\rG(n-2,V)$ over $[V_{n-2}]$. 
It also implies that 
$\Lrho_{\hcoY_{2}}^{\;*}\sQ$ is trivial on a fiber of 
$\tLrho_{\hcoY_{2}}\colon \hcoY_2\to \widetilde{\hcoY}$
by \cite[Prop.~4.22]{Geom}.
The last property and Lemma \ref{lem:desbundle} below ensure
the existence of a locally free sheaf $\widetilde{\sQ}$ on $\widetilde{\hcoY}$
such that 
$\Lrho_{\hcoY_{2}}^{\;*}\sQ=\tLrho_{\hcoY_{2}}^{\;*}\widetilde{\sQ}$.
\end{proof}

The following lemma should be well-known for experts but we give a proof
for readers' convenience:
\begin{lem}
\label{lem:desbundle}
Let $Y$ be a smooth projective variety and $f\colon X\to Y$ the blow-up
along a smooth subvariety $S\subset Y$.
Let $\sE$ be a locally free sheaf on $X$ of rank $r$ such that,
for any nontrivial fiber $\gamma$ of $f$, it holds that 
$\sE|_{\gamma}\simeq \sO_{\gamma}^{\oplus r}$. Then there exists a locally free sheaf $\overline{\sE}$ on $Y$
such that $\sE=f^*\overline{\sE}$.
\end{lem}

\begin{proof}
We give a proof by using Mori theory.
Let $\pi\colon \mP(\sE)\to X$ be the natural projection
and $H$ the tautological divisor.
We denote by $E$ the $f$-exceptional divisor and
set $F:=\pi^{-1}(E)$.
Since
$\sE|_{\gamma}\simeq \sO_{\gamma}^{\oplus r}$,
we have $\pi^{-1}(\gamma)\simeq \mP^{r-1}\times \gamma$ and 
$H|_{\pi^{-1}(\gamma)}$ is a divisor of type $(1,0)$.
Moreover, since $K_{\mP(\sE)}=-rH+\pi^*(\det \sE +K_X)$,
any curve in a fiber of $\pi^{-1}(\gamma)\to \mP^{r-1}$
is negative for $K_{\mP(\sE)}$.
Therefore any curve in a fiber of $\pi^{-1}(\gamma)\to \mP^{r-1}$
spans a $K_{\mP(\sE)}$-negative extremal ray in $\overline{\mathrm{NE}}(\mP(\sE)/Y)$
and a sufficient multiple of $H$ defines a birational morphism $p\colon \mP(\sE)\to \overline{\mP}$ over $Y$ contracting 
fibers of $\pi^{-1}(\gamma)\to \mP^{r-1}$.
Let $q\colon \overline{\mP}\to Y$ be the induced morphism.
It is easy to see that $\overline{\mP}$ is smooth and $p$ is the blow-up along
$p(F)$.
Moreover, $H$ is the pull-back of a Cartier divisor $\overline{H}$ on $\overline{\mP}$ and the restriction of $\overline{H}$ to any fiber $\simeq \mP^{r-1}$ of $\overline{\mP}\to Y$ is $\sO_{\mP^{r-1}}(1)$.
Therefore $\overline{\sE}:=q_*\sO_{\overline{\mP}}(\overline{H})$ is a locally free sheaf of rank $r$ on $Y$ and it holds that $\sE\simeq f^* \overline{\sE}$ by construction.  
\end{proof}

\subsection{Locally free sheaf $\widetilde{\frQ}$ on $\widetilde{\hcoY}$}

Let us focus on the local geometry of the blow-up $\hcoY_{3}\to\overline{\hcoY}'$
which is described by $\Prt_{\rho}=\rF(n-3,n-2,V)\to\overline{\Prt}_{\rho}=\rG(n-2,V)$. We denote
the universal sub-bundles of the partial flag variety $\rF(n-3,n-2,V)$
by $\sR_{n-3}\subset\sR_{n-2}\subset\sR_{V}$, where we set $\sR_{V}:=V\otimes\sO_{F(n-3,n-2,V)}$
and $\rank\sR_{k}=k$. There is an exact sequence \begin{equation}
0\to\sR_{n-2}/\sR_{n-3}\to\sR_{V}/\sR_{n-3}\to\sR_{V}/\sR_{n-2}\to0.\label{eq:exact-sequence-Rv}\end{equation}
It is also useful to identify $\Prt_{\rho}$ with the projective
bundle $\mP(\frQ)$ over $\rG(n-3,V)$. Then the exact sequence above
is nothing but the relative Euler sequence of the projective bundle
$\pi_{\sP}:\mP(\frQ)\to\rG(n-3,V)$ with\[
\sR_{n-2}/\sR_{n-3}=\sO_{\mP(\,\frQ\,\,)}(-1),\,\sR_{V}/\sR_{n-3}=\pi_{\sP}^{*}\frQ,\]
and $\sR_{V}/\sR_{n-2}$ is the twisted relative tangent bundle 
$T_{\Prt_{\rho}/\rG(n-2,V)}(-1)$.

\begin{prop}
Let $\Lpi_{\hcoY_{2}}=\Lpi_{\hcoY_{3}}\circ\Lrho_{\hcoY_{2}}$ be
the composite of $\Lrho_{\hcoY_{2}}:\hcoY_{2}\to\hcoY_{3}$ with
$\Lpi_{\hcoY_{3}}:\hcoY_{3}\to\rG(n-3,V)$. Denote by $i:F_{\rho}\hookrightarrow\hcoY_{2}$
the inclusion of the exceptional divisor, and by $\Lrho_F:F_{\rho}\to\Prt_{\rho}\subset\hcoY_{3}$
the restriction $\Lrho_F:=\Lrho_{\hcoY_{2}}\vert_{F_{\rho}}$. Then
the kernel \[
\frR:=\Ker\left\{ \Lpi_{\hcoY_{2}}^{\;*}\frQ^{*}\to i_{*}\circ\Lrho_F^{*}
\sO_{\mP(\,\frQ\,\,)}(1)\right\} \]
is a locally free sheaf on $\hcoY_{2}$. 
\end{prop}
\begin{proof}From (\ref{eq:exact-sequence-Rv}), we have a surjection
$\Lrho_F^{*}(\sR_{V}/\sR_{n-3})^{*}\to\Lrho_F^{*}(\sR_{n-2}/\sR_{n-3})^{*}\to0$\textcolor{red}{{}
}and also $i_{*}\circ\Lrho_F^{*}(\sR_{V}/\sR_{n-3})^{*}\to i_{*}\circ\Lrho_F^{*}(\sR_{n-2}/\sR_{n-3})^{*}\to0$.
Let $\frR$ be the kernel of the composite of the latter and
the natural surjection $\Lpi_{\hcoY_{2}}^{\;*}\frQ^{*}\to i_{*}\circ\Lrho_F^{*}\frQ^{*}=i_{*}\circ\Lrho_F^{*}(\sR_{V}/\sR_{n-3})^{*}$.
Then we obtain the exact sequence \begin{equation}
0\to\frR\to\Lpi_{\hcoY_{2}}^{\;*}\frQ^{*}\to i_{*}\circ\Lrho_F^{*}
\sO_{\mP(\,\frQ\,\,)}(1)\to0.\label{eq:exact-seq-T2}\end{equation}
 By taking $\sE xt(-,\sO_{\hcoY_{2}})$ of this sequence, we see that
$\frR$ is a locally free sheaf on $\hcoY_{2}$ (see \cite[III, Ex 6.6]{Ha}).
\end{proof}
\begin{lem}
\label{lem:T2-sum} $\frR\vert_{\delta}=\sO_{\delta}^{\oplus4}$
for each fiber $\delta\simeq\mP^{n-3}$ of $F_{\rho}\to G_{\rho}$.
\end{lem}

\begin{proof}Each fiber $\delta$ of $F_{\rho}\to G_{\rho}$ projects
isomorphically to a fiber of $F(n-3,n-2,V)\to\rG(n-2,V)$, and further to
a copy of $\mP^{n-3}$ in $\rG(n-3,V)$. Therefore $\Lpi_{\hcoY_{2}}^{\;*}\frQ^{*}\vert_{\delta}\simeq\sO_{\mP^{n-3}}^{\oplus3}\oplus\sO_{\mP^{n-3}}(-1)$
and also $\Lrho_F^{*}\sO_{\mP(\,\frQ\,\,)}(1)\vert_{\delta}\simeq\sO_{\mP^{n-3}}(-1)$.
By restricting the exact sequence (\ref{eq:exact-seq-T2}), we obtain
\[
\frR\vert_{\delta}\to\sO_{\mP^{n-3}}^{\oplus3}\oplus\sO_{\mP^{n-3}}(-1)\to\sO_{\mP^{n-3}}(-1)\to0,\]
which shows that there is a surjection $\frR\vert_{\delta}\to\sO_{\mP^{n-3}}^{\oplus3}$
with its kernel being an invertible sheaf $\sL$. Note that $\det\frR\simeq\sO_{\hcoY_{2}}(-L_{\hcoY_{2}}-F_{\rho})$
from (\ref{eq:exact-seq-T2}). Now, since $L_{\hcoY_{2}}\vert_{\delta}=\sO_{\delta}(1)$
by definition and also $F_{\rho}|_{\delta}=\sO_{\mP^{n-3}}(-1)$, we have $\det\,\frR\vert_{\delta}\simeq\sO_{\delta}$
and $\sL\simeq\sO_{\delta}$. Hence $\frR\vert_{\delta}\simeq\sO_{\delta}^{\oplus4}$.
\end{proof}

Now we define $\frQ_2:=\frR^{*}$. From Lemmas \ref{lem:desbundle} and 
\ref{lem:T2-sum}, we have the following proposition (and definition):
\begin{prop}
\label{def:tidelT} There exists a locally free sheaf $\widetilde{\frQ}$
on $\widetilde{\hcoY}$ such that \[
\frQ_2=\tLrho_{\hcoY_{2}}^{\;*}\widetilde{\frQ}.\]
 
\end{prop}
The following exact sequence will be used in our later calculations:
\begin{prop}
There exists the following exact sequence$:$

\begin{equation}
0\to\Lpi_{\hcoY_{2}}^{\;*}\frQ\to\frQ_2\to i_{*}\circ\Lrho_F^{*}\sO_{\mP(\,\frQ\,\,)}(-1)({F_{\rho}}|_{F_{\rho}})\to0.\label{eq:Eb}\end{equation}
\end{prop}
\begin{proof}
By taking $\mathcal{H}om(-,\sO_{\hcoY_{2}})$ of (\ref{eq:exact-seq-T2}), we obtain: \[
0\to\Lpi_{\hcoY_{2}}^{\;*}\frQ\to\frQ_2\to\sE xt_{\sO_{\hcoY_{2}}}^{1}(i_{*}\circ\Lrho_F^{*}\sO_{\mP(\,\frQ\,\,)}(1),\sO_{\hcoY_{2}})\to0.\]
 The claim follows by evaluating $\sE xt_{\sO_{\hcoY_{2}}}^{1}$ as
\[
\sE xt_{\sO_{\hcoY_{2}}}^{1}\big(i_{*}\circ\Lrho_F^{*}\sO_{\mP(\,\frQ\,\,)}(1),\sO_{\hcoY_{2}}\big)\simeq i_{*}\sE xt_{\sO_{F_{\rho}}}^{0}\big(\Lrho_F^{*}\sO_{\mP(\,\frQ\,\,)}(1),i^{*}\sO_{\hcoY_{2}}\otimes\omega_{F_{\rho}/\hcoY_{2}}),\]
where we use the Grothendieck-Verdier duality with $i^{!}\sQ_{\hcoY_{2}}=i^{*}\sO_{\hcoY_{2}}\otimes\omega_{F_{\rho}/\hcoY_{2}}[-1]=\sO_{F_{\rho}}(F_{\rho}\vert_{F_{\rho}})[-1]$.\textcolor{red}{{} }
\end{proof}
\vspace{0.5cm}

\subsection{Properties of $\sS^{*},\sQ$ restricted on $\Prt_{\rho}$ and $\Prt_{\sigma}$}

As in the last subsection, we identify $\Prt_{\rho}=\rF(n-3,n-2,V)$ with
the projective bundle $\mP(\frQ)$ with $\Lpi_{\sP}:\mP(\frQ)\to\rG(n-3,V)$.
We introduce two divisors on $\Prt_{\rho}$;\[
H_{\Prt_{\rho}}=\sO_{\mP(\frQ\ )}(1)\text{\;\;{and}}\;\; L_{\Prt_{\rho}}:=\pi_{\sP}^{*}\sO_{\rG(n-3,V)}(1).\]

\begin{prop}
\label{cla:mislead} $\sQ|_{\Prt_{\rho}}\simeq\sS^{*}(L_{\hcoY_{3}})|_{\Prt_{\rho}}$
and $\sQ|_{\Prt_{\sigma}}\simeq\sS^{*}(L_{\hcoY_{3}})|_{\Prt_{\sigma}}$. 

\end{prop}
\begin{proof} 
Proofs of the both relations are similar, so we only prove the former. 
Take a point $[\rm{P}_{V_{n-2}/V_{n-3}}]$ of 
$\Prt_{\rho} \;\subset \mathrm{G}(3,\wedge^2 \frQ)$.
Let $W_1$ and $W_2$ be the fiber of 
$\sS^*$ and $\sQ$ at $[\rm{P}_{V_{n-2}/V_{n-3}}]$,
respectively. 
We compare the restrictions of
the universal exact sequence (\ref{eq:univ}) and its dual;
\[
0\to W_1 \to \wedge^2 V/V_{n-3} \to W_2 \to 0,
\]
\[
0\to (W_2)^* \to (\wedge^2 V/V_{n-3})^*\to (W_1)^* \to 0.
\]
Note that $\mP(W_1)=\rm{P}_{V_{n-2}/V_{n-3}}$.
Choose 
a basis $\{\bar{\bf e}_1, \bar{\bf e}_2, \bar{\bf e}_3, \bar{\bf e}_4\}$
of $V/V_{n-3}$ so that 
$V_{n-2}/V_{n-3} =\langle \bar{\bf e}_1\rangle$.
Then $W_1= 
\langle \bar{\bf e}_1\wedge \bar{\bf e}_2,
\bar{\bf e}_1\wedge \bar{\bf e}_3,
\bar{\bf e}_1\wedge \bar{\bf e}_4\rangle$.
By the non-degenerate pairing
$\wedge^2 V/V_{n-3}\times \wedge^2 V/V_{n-3}\to \wedge^4 V/V_{n-3}\simeq \mC$,
we may identify $\wedge^2 V/V_{n-3}$ and $(\wedge^2 V/V_{n-3})^*$.
Under this identification, we see from 
the explicit basis of $W_1$ that 
$(W_2)^*$ coincides with $W_1$
since an element of $(W_2)^*$ is nothing but 
an element of $(\wedge^2 V/V_{n-3})^*$ vanishing on $W_1$.
Since $\wedge^4 V/V_{n-3}$ is a fiber of $L_{\Prt_{\rho}}$ at 
$[\rm{P}_{V_{n-2}/V_{n-3}}]$,
we conclude that 
$\sS|_{\Prt_\rho}\simeq \sQ^*|_{\Prt_\rho}\otimes L_{\Prt_\rho}$.
\end{proof}
\begin{prop}
\label{cla:det} $\;$ 

\noindent {\rm (1)} $\sQ|_{\Prt_{\rho}}=\wedge^{2}(\sR_{V}/\sR_{n-2})\simeq(\sR_{V}/\sR_{n-2})^{*}(H_{\Prt_{\rho}}+L_{\Prt_{\rho}})$. 

\noindent {\rm (2)} $\det\sQ|_{\Prt_{\rho}}=2(H_{\Prt_{\rho}}+L_{\Prt_{\rho}})$.

\end{prop}
\begin{proof}(1) Since 
$\sQ\vert_{[V_{n-3},V_{n-2}]}=\wedge^{2}(V/V_{n-2})$ by the proof of 
Proposition \ref{def:SQT1}, and
$(\sR_{V}/\sR_{n-2})\vert_{[V_{n-3},V_{n-2}]}=V/V_{n-2}$, we have $\sQ\vert_{\Prt_{\rho}}=\wedge^{2}(\sR_{V}/\sR_{n-2})$.
By taking the determinant of (\ref{eq:exact-sequence-Rv}), we have
\[
\wedge^{3}(\sR_{V}/\sR_{n-2})\simeq\wedge^{4}(\sR_{V}/\sR_{n-3})\otimes(\sR_{n-2}/\sR_{n-3})^{*}\simeq L_{\Prt_{\rho}}\otimes\sO_{\mP(\frQ\,)}(1).\]
Therefore we obtain the isomorphism $\sQ\vert_{\Prt_{\rho}}=\wedge^{2}(\sR_{V}/\sR_{n-2})\simeq(\sR_{V}/\sR_{n-2})^{*}(H_{\Prt_{\rho}}+L_{\Prt_{\rho}}).$
The claim (2) follows from $\det\sQ\vert_{\Prt_{}}=\wedge^{3}(\wedge^{2}(\sR_{V}/\sR_{n-2}))\simeq(\wedge^{3}(\sR_{V}/\sR_{n-2}))^{\otimes2}$.
\end{proof}

\section{{\bf Lefschetz collection in $\sD^{b}(\widetilde{\hcoY})$}}
\label{section:Lefschetz-Y}

In this section, we assume that $n=3,4$.
Using the sheaves $\widetilde{\sS}_{L},\widetilde{\sQ},\,\widetilde{\frQ}$
introduced in Section \ref{sec:SheavesDef}, we construct
a Lefschetz collection in $\sD^{b}(\widetilde{\hcoY})$, which shows
an interesting duality between the (dual) Lefschetz collection obtained
in Theorem \ref{thm:Gvan1}. 

\subsection{Homological properties of $\widetilde{\sS}_{L},\widetilde{\sQ},\,\widetilde{\frQ}$}

\begin{thm}
\label{thm:digGvan}
Suppose $n=3,4$.
\begin{enumerate}[$(1)$]
\item
The ordered sequence
$\sO_{\widetilde{\hcoY}}$,
$\widetilde{\frQ}$, 
$\widetilde{\sS}^*_L$, or
$\widetilde{\sQ}$
is semi-orthogonal.
\item
Let
$\widetilde{\sA}$ or $\widetilde{\sB}$ be one of
the locally free sheaves 
$\sO_{\widetilde{\hcoY}}$,
$\widetilde{\frQ}$, 
$\widetilde{\sS}^*_L$, or
$\widetilde{\sQ}$ on $\widetilde{\hcoY}$.
Then 
\[
H^{\bullet}(\widetilde{\sA}^*\otimes \widetilde{\sB} (-t))=0 \ \text{for}\ 
\begin{cases}
1\leq t\leq 5: n=3,\\
t=6,7: n=3 \ \text{and} \ \widetilde{\sA}=\sO_{\widetilde{\hcoY}},\\
1\leq t \leq 9: n=4.
\end{cases}
\]
\end{enumerate}
\end{thm}

 The rest of this section is devoted to our proof of Theorem \ref{thm:digGvan},
where we compute the cohomology groups $H^{\bullet}(\widetilde{\sA}^*\otimes \widetilde{\sB} (-t))$ for $n=3$ and $1\leq t\leq 7$ and for
$n=4$ and $1\leq t\leq 9$. 
Our strategy is to reduce the computations of cohomology
groups on $\widetilde{\hcoY}$ to those on $\hcoY_{3}$ and use 
Theorem \ref{thm:Bott} for the $\mathrm{G}(3,6)$-bundle $\hcoY_{3}\to\rG(n-3,V)$.

Let $\widetilde{\frak{F}}$ be a locally free sheaf on $\widetilde{\hcoY}$
and $\frak{F}_2:=\tLrho_{\hcoY_{2}}^{\;*}\widetilde{\frak{F}}$. Since
$\tLrho_{\hcoY_{2}}\colon\hcoY_{2}\to\widetilde{\hcoY}$ is a blow-up
of a smooth variety, it holds that \begin{equation}
H^{\bullet}(\widetilde{\hcoY},\widetilde{\frak{F}}(-t))\simeq H^{\bullet}(\hcoY_{2},{\frak{F}}_2(-t)),\label{eq:rhs}\end{equation}
 where $(-t)$ on the right hand side represents the twist by $\sO_{\widetilde{\hcoY_{2}}}(-tM_{\widetilde{\hcoY_{2}}})$.
Therefore it suffices to compute
$H^{\bullet}({\sA_2}^*\otimes {\sB_2} (-t))$.

\subsection{Divisors on $\hcoY_2$}

Recall the universal sequence (\ref{eq:univ}) of the Grassmann bundle $\hcoY_{3}=\rG(3,\wedge^2 \mathfrak{Q})$.
Taking the determinant and using $\wedge^{6}(\wedge^2 \mathfrak{Q})=
\sO_{\hcoY_3}(3)$,
we have \begin{equation}
\det\sQ=\det\sS^{*}+3L_{\hcoY_{3}}=\det\{\sS^{*}(L_{\hcoY_{3}})\}.\label{eq:detdiff}\end{equation}
 Also, since $T_{\hcoY_{3}/\mP(V)}=\sS^{*}\otimes\sQ$ (see \cite[p.435]{Ful}),
we have \begin{equation}
K_{\hcoY_{3}}=-\det(\sQ\otimes\sS^{*})-(n+1)L_{\hcoY_{3}}=-3(\det\sQ+\det\sS^{*})-(n+1)L_{\hcoY_{3}}=-6\det\sQ+(8-n)L_{\hcoY_{3}},\label{eq:KY3}\end{equation}
 where we note $\rank\sS=\rank\sQ=3$ and we use (\ref{eq:detdiff})
in the last equality.

We now see some relations among divisors on $\hcoY_{2}$. Note that
\begin{equation}
K_{\hcoY_{2}}=\Lrho_{\hcoY_{2}}^{\;*}K_{\hcoY_{3}}+5F_{\rho}\label{eq:adj}\end{equation}
 since $\Lrho_{\hcoY_{2}}$ is the blow-up along a smooth subvariety
of codimension $6$. By this and (\ref{eq:KY3}), we have \begin{equation}
K_{\hcoY_{2}}=-6\Lrho_{\hcoY_{2}}^{\;*}\det\sQ+(8-n)L_{\hcoY_{2}}+5F_{\rho}.\label{eq:canY2}\end{equation}

\begin{prop}
 \label{cla:M} The pull-back $M_{\hcoY_{2}}$ of $\sO_{\Hes}(1)$
is given by \[
M_{\hcoY_{2}}=\Lrho_{\hcoY_{2}}^{\;*}(\det\sQ)-L_{\hcoY_{2}}-F_{\rho}.\]
 
\end{prop}
\begin{proof} 
Note that 
${\Lpi_{\hcoY_3}}_*\sO_{\mP(\wedge^3 (\wedge^2 \frQ\,\,))}(1)$ 
$=\wedge^3 (\wedge^2 \frQ^*)$, and  
${\Lpi_{\UU}}_*\sO_{\mP({\sf S}^2 \frQ^{\,\,*})}(1)$ 
$={\sf S}^2 \frQ$.
Therefore, by the decomposition \cite[(4.5)]{Geom}
and the construction of $\hcoY_3\dashrightarrow \UU$
as a relative linear projection,
we have
\[
\Lrho_{\hcoY_2}^{\;*}(\det \sS^*)-F_{\rho}
=q_{\hcoY_2}^{\;*}(M_{\UU}-2L_{\UU})=
M_{\hcoY_2}-2L_{\hcoY_2}.
\]
Then
we have the assertion
by (\ref{eq:detdiff}).\end{proof}

\begin{prop}
\label{cor:weakFano}
$\hcoY_2$ is a weak Fano manifold, namely, $-K_{\hcoY_2}$ is nef and big.
\end{prop}

\begin{proof}
Note that $\det \sQ$ is nef since $\sQ$ is the image of 
the surjection from 
the globally generated bundle $\Lpi_{\hcoY_3}^*{\wedge^2 \mathfrak{Q}}$.
By (\ref{eq:canY2}) and Proposition \ref{cla:M},
we have
$-K_{\hcoY_2}=5M_{\hcoY_2}+(n-3)L_{\hcoY_2}+\Lrho_{\hcoY_2}^* \det \sQ$,
which is clearly nef, and is also big since so is 
$M_{\hcoY_2}$.
\end{proof}

\begin{prop}
\label{prop:F'}
Let $F'_{\widetilde{\hcoY}}$ be the strict transform of $F_{\widetilde{\hcoY}}$.
It holds that $F'_{\widetilde{\hcoY}}=2M_{\hcoY_2}-L_{\hcoY_2}-F_{\rho}$.
\end{prop}

\begin{proof}
By \cite[Prop.~2.5, Cor.~5.2]{Geom},
we have 
\begin{equation}
\label{eq:YY2}
K_{\widetilde{\hcoY}}=\Lrho_{\widetilde{\hcoY}}^*K_{\hcoY}+(n-2)F_{\widetilde{\hcoY}}=-2(n+1)M_{\widetilde{\hcoY}}+(n-2)F_{\widetilde{\hcoY}}.
\end{equation}
By \cite[Prop.~4.21]{Geom},
$\widetilde{\Lrho}_{\hcoY_2}\colon \hcoY_2\to \widetilde{\hcoY}$ is the blow-up along $G_{\rho}$.
Therefore 
\begin{equation}
\label{eq:blup}K_{\hcoY_2}=\widetilde{\Lrho}_{\hcoY_2}^*K_{\widetilde{\hcoY}}+(n-3)F_{\rho}.
\end{equation}
Since $G_{\rho}$ is not contained in $F_{\widetilde{\hcoY}}$,
we deduce from (\ref{eq:YY2}) and (\ref{eq:blup}) that
\[
K_{\hcoY_2}=-2(n+1)M_{{\hcoY}_2}+(n-2)F'_{\widetilde{\hcoY}}+(n-3)F_{\rho}.
\]
Combining this with (\ref{eq:canY2}) and Proposition \ref{cla:M},
we obtain
\[
(n-2)F'_{\widetilde{\hcoY}}=(2n-4)M_{\hcoY_2}-(n-2)L_{\hcoY_2}-(n-2)F_{\rho}.
\]
Since $\hcoY_2$ is a weak Fano manifold, $\Pic \hcoY_2$ is torsion-free.
Therefore we obtain the desired equality.

\end{proof}

\subsection{Case $\widetilde{\sA}=\widetilde{\frQ}$ or $\widetilde{\sB}=\widetilde{\frQ}$} 
Among 
$\sO_{{\hcoY}_2}$,
${\frQ}_2$, 
$\Lrho_{\hcoY_2}^*{\sS}(-L)$, and
$\Lrho_{\hcoY_2}^*{\sQ}$,
only ${\frQ}_2$ is not the pull-back of 
a locally free sheaf on $\hcoY_3$.
We will show that to compute $H^{\bullet}(\hcoY_2, \sA^*_2\otimes \sB_2(-t))$ for $1\leq t\leq 5$ and $\sA_2=\frQ_2$ or $\sB_2=\frQ_2$,
we may replace $\frQ_2$ by $\Lrho_{\hcoY_2}^*\frQ$.

\begin{lem}
\label{cla:Tn=4} 
With the notation as in Theorem $\ref{thm:digGvan}$,
it holds that
\begin{align}
H^{\bullet}(\frQ_2^*\otimes {\sB_2} (-t))\simeq 
H^{\bullet}(\Lpi_{\hcoY_{2}}^{\;*}\frQ^*\otimes {\sB_2} (-t)),\label{align:a}\\
H^{\bullet}(\sA_2^*\otimes \frQ_2 (-t))\simeq 
H^{\bullet}(\sA_2^*\otimes \Lpi_{\hcoY_{2}}^{\;*}\frQ (-t)),\label{align:b}\\
H^{\bullet}(\frQ_2^*\otimes \frQ_2 (-t))\simeq 
H^{\bullet}(\Lpi_{\hcoY_{2}}^{\;*}\frQ^*\otimes \Lpi_{\hcoY_{2}}^{\;*}\frQ (-t))\label{align:c}
\end{align}
for any $\bullet$ and $1\leq t\leq 5$.
Moreover, $(\ref{align:a})$ holds also for $t=0$ and $\sB_2=\sO_{\hcoY_2}$.
\end{lem}

\begin{proof}~
 
\noindent {\bf Proof of (\ref{align:a}).}
By the exact sequence (\ref{eq:exact-seq-T2}), we have 
\begin{equation}
\label{eq:basic-a}
0\to\frQ_2^{*}\otimes\sB_2(-t)\to\Lpi_{\hcoY_{2}}^{\;*}\frQ^{*}\otimes\sB_2(-t)\to i_{*}\Lrho_F^{*}\sO_{\mP(\,\frQ\,\,)}(1)\otimes\sB_2(-t)|_{F_{\rho}}\to0.
\end{equation}
Therefore it suffices to show 
\begin{equation}
\label{eq:vanB_2}
\ensuremath{H^{\bullet}(F_{\rho},\Lrho_F^{*}\sO_{\mP(\,\frQ\,\,)}(1)\otimes
\sB_2(-t)|_{F_{\rho}})=0}.
\end{equation}
Note that this will imply 
for $\sB_2=\frQ_2$ that
\begin{equation} 
\label{eq:c}
H^{\bullet}(\frQ_2^*\otimes \frQ_2 (-t))\simeq 
H^{\bullet}(\Lpi_{\hcoY_{2}}^{\;*}\frQ^*\otimes \frQ_2(-t))\end{equation}

\noindent {\bf Case $n=4$.} In this case,
this vanishing holds for any $t$ by the Leray spectral sequence for 
$\tLrho_{\hcoY_{2}}|_{F_{\rho}}\colon F_{\rho}\to G_{\rho}$,
since $\tLrho_{\hcoY_{2}}|_{F_{\rho}}$ is a $\mP^{1}$-bundle and
the restriction of $\Lrho_F^{*}\sO_{\mP(\,\frQ\,\,)}(1)\otimes\sB_2(-t)|_{F_{\rho}}$
to the fiber is a direct sum of $\sO_{\mP^{1}}(-1)$ by the proof
of Lemma \ref{lem:T2-sum}. In particular, 
$(\ref{align:a})$ holds also for $t=0$ and $\sB_2=\sO_{\hcoY_2}$ in this case.

\vspace{5pt}

\noindent {\bf Case $n=3$.}
If  
$\sB_2=\sO_{{\hcoY}_2}$,
$\Lrho_{\hcoY_2}^*{\sS}(-L)$, and
$\Lrho_{\hcoY_2}^*{\sQ}$,
then the vanishing (\ref{eq:vanB_2}) holds
for $1\leq t\leq 5$
by the Leray spectral sequence for 
$\Lrho_F\colon 
F_{\rho}\to \Prt_{\rho}$ since 
$\Lrho_F$
is a $\mP^5$-bundle and
the restriction of 
$\Lrho_F^*\sO_{\mP(V)}(1)\otimes \sB_2(-t)|_{F_{\rho}}$
to its fiber is a direct sum of $\sO_{\mP^5}(-t)$.

Suppose $t=0$ and $\sB_2=\sO_{\hcoY_2}$.
Then (\ref{align:a}) holds by (\ref{eq:basic-a}) since $H^{\bullet}(V^*\otimes \sO_{\hcoY_2})$ and $H^{\bullet}(
\Lrho_F^* 
\sO_{\mP(V)}(1))$ are isomorphic to each other for $\bullet=0$
and are zero for $\bullet>0$.

Suppose $\sB_2=\frQ_2$. 
First we calculate the restriction of
$\frQ_2^*$ on $F_{\rho}$.
By restricting (\ref{eq:exact-seq-T2}) to $F_{\rho}$,
we obtain the exact sequence
\[
\frQ_2^*|_{F_{\rho}}\to V^*\otimes \sO_{F_{\rho}}\to 
\Lrho_F^* 
\sO_{\mP(V)}(1)\to 0.
\]
Since
the kernel of
$V^*\otimes \sO_{F_{\rho}}\to 
\Lrho_F^* 
\sO_{\mP(V)}(1)$ is isomorphic to 
$\Lrho_F^*\Omega_{\mP(V)}(1)$,
we have the exact sequence
\begin{equation}
\label{eq:extT*}
0\to \sL\to \frQ_2^*|_{F_{\rho}}\to 
\Lrho_F^*\Omega_{\mP(V)}(1)\to 0,
\end{equation}
where $\sL$ is an invertible sheaf on $F_{\rho}$.
Taking the determinants in (\ref{eq:exact-seq-T2}), we have
$\det \frQ_2^*=\sO_{\widetilde{\hcoY}}(-F_{\rho})\otimes \wedge^4 V^*$. 
Therefore
$\sL\simeq \sO_{F_{\rho}}(-F_{\rho})\otimes 
\Lrho_F^* 
\sO_{\mP(V)}(1)\otimes \wedge^4 V^*$.
Now dualizing (\ref{eq:extT*}), we obtain
\begin{equation}
\label{eq:extT}
0\to 
\Lrho_F^*T_{\mP(V)}(-1)
\to {\frQ}_2|_{F_{\rho}}\to
\sO_{F_{\rho}}(F_{\rho})\otimes 
\Lrho_F^*
\sO_{\mP(V)}(-1) \otimes \wedge^4 V
\to 0.
\end{equation}
Let $P$ be a fiber of
$\Lrho_F\colon F_{\rho}\to \Prt_{\rho}\simeq \mP(V)$.
We have the vanishing (\ref{eq:vanB_2}) for $1\leq t\leq 4$
since the restriction to $P$ of
$\frQ_2|_{F_{\rho}}$ is isomorphic to
$\sO_P^{\oplus 3}\oplus \sO_P(-1)$ by (\ref{eq:extT}).
Finally we consider the case where $t=5$.
By the Serre duality,
it suffices to show
the vanishing of
\[
H^{8-\bullet}(F_{\rho}, 
\Lrho_F^*\sO_{\mP(V)}(-1)
\otimes \frQ_2^*(-\det \sQ_2+F_{\rho})|_{F_{\rho}}),
\]
which can be seen by twisting
(\ref{eq:extT*}) with
$\Lrho_F^*\sO_{\mP(V)}(-1)
\otimes \sO_{F_{\rho}}(-\det {\sQ}_2+F_{\rho})$.

\noindent {\bf Proof of (\ref{align:b}) and (\ref{align:c}).}
By the exact
sequence (\ref{eq:Eb}), we have \[
0\to\Lpi_{\hcoY_{2}}^{\;*}\frQ\otimes{\frak{F}}(-t)\to\frQ_2\otimes{\frak{F}}(-t)\to i_{*}\Lrho_F^*\sO_{\mP(\,\frQ\,\,)}(-1)\otimes{\frak{F}}(-t{M}_{\hcoY_{2}}+F_{\rho})|_{F_{\rho}}\to0\]
 for a locally free sheaf $\frak{F}$ on $\hcoY_{2}$. It suffices to show the
vanishings of \begin{equation}
H^{\bullet}(F_{\rho},\Lrho_F^{*}\sO_{\mP(\,\frQ\,\,)}(-1)\otimes{\frak{F}}(-t{M}_{\hcoY_{2}}+F_{\rho})|_{F_{\rho}})\;\;(1\leq t\leq5)\label{eq:Frho}\end{equation}
with ${\frak{F}}=\sO_{\hcoY_{2}}$, $\Lrho_{\hcoY_{2}}^{\;*}\sQ^{*}$,
$\Lrho_{\hcoY_{2}}^{\;*}\sS(-L_{\hcoY_{2}})$
for (\ref{align:b}), and with $\Lpi_{\hcoY_{2}}^{\;*}\frQ^{*}$
for (\ref{align:c}) in view of (\ref{eq:c}).
Since $\sQ^{*}\vert_{\sP_{\rho}}\simeq\sS(-L_{\hcoY_{3}})\vert_{\sP_{\rho}}$
by Proposition \ref{cla:mislead}, we have only to consider the cases
${\frak{F}}=\sO_{\hcoY_{2}}$, $\Lrho_{\hcoY_{2}}^{\;*}\sQ^{*},$
$\Lpi_{\hcoY_{2}}^{\;*}\frQ^{*}$. For $1\leq t\leq4$, the vanishings
of (\ref{eq:Frho}) follow from the Leray spectral sequence for $\Lrho_F\colon F_{\rho}\to\Prt_{\rho}$
since $\Lrho_F$ is a $\mP^{5}$-bundle and the restriction of $\Lrho_F^{*}
\sO_{\mP(\,\frQ\,\,)}(-1)\otimes{\frak{F}}(-t{M}_{\hcoY_{2}}+F_{\rho})|_{F_{\rho}}$
to the fiber is a direct sum of $\sO_{\mP^{5}}(-(t+1))$ by Proposition
\ref{cla:M}. For $t=5$, note that (\ref{eq:Frho}) is Serre dual
to \begin{equation}
H^{\bullet'}(F_{\rho},\Lrho_F^{*}\sO_{\mP(\,\frQ\,\,)}(1)\otimes{\frak{F}}^{*}(\rho_{\hcoY_{2}}^{*}(-\det\sQ-L_{\hcoY_{3}}))|_{F_{\rho}})\label{eq:dammy}\end{equation}
with $\bullet'=8-\bullet$ for $n=3$ and $\bullet'=12-\bullet$ for $n=4$ 
by (\ref{eq:canY2}) and Proposition \ref{cla:M}. Since $\Lrho_{\hcoY_{2}}$
is the blow-up of a smooth variety and $\frak{F}=\Lrho_{\hcoY_{2}}^{\;*}{\frak{F}}_3$
with a locally free sheaf ${\frak{F}}_3$ on $\hcoY_{3}$,
each of (\ref{eq:dammy}) is isomorphic to \begin{equation}
H^{\bullet'}(\Prt_{\rho},\Lrho_F^*\sO_{\mP(\,\frQ\,\,)}(1)\otimes\frak{F}_3^{*}(-\det\sQ-L_{\hcoY_{3}})|_{\Prt_{\rho}}).\label{eq:dammy2}\end{equation}
 Using Proposition \ref{cla:det} (1) and (2), we can write \[
\sO_{\mP(\,\frQ\,\,)}(1)\otimes\frak{F}_3^{*}(-\det\sQ-L_{\hcoY_{3}})|_{\Prt_{\rho}}=\begin{cases}
\sO_{\Prt_{\rho}}(-H_{\Prt_{\rho}}\text{-- }3L_{\Prt_{\rho}})\text{ for }\frak{F}_3=\sO_{\hcoY_{3}}\\
(\sR_{V}/\sR_{n-2})^{*}(-2L_{\Prt_{\rho}})\text{ for }\frak{F}_3=\sQ^{*}\\
\pi_{\sP}^*\frQ(-H_{\Prt_{\rho}}\text{-- }3L_{\Prt_{\rho}})\text{ for }
\frak{F}_3=\Lpi_{\hcoY_{3}}^{\;*}\frQ^{*}.\end{cases}\]
We see that all $H^{\bullet'}(\Prt_{\rho},-)$'s of these sheaves vanish
when restricted to fibers of $\pi_{\sP}\colon\Prt_{\rho}=\mP(\frQ)\to\rG(n-3,V)$.
Hence, by the Leray spectral sequence for $\pi_{\sP}$, all of (\ref{eq:dammy2})\textcolor{red}{{}
}vanish, too. \end{proof}

\subsection{Case $0\leq t\leq 5$}

By using the following lemma, we can reduce our computations
of cohomology groups to those on $\hcoY_{3}$ in the case where $0\leq t\leq 5$:
\begin{lem}
\label{cla:van2} $\;$

\noindent{\rm (1)} $R^{q}\Lrho_{\hcoY_{2}*}\sO_{\hcoY_{2}}(tF_{\rho})=0$
for any $t\leq5$ and $q>0$. 

\noindent{\rm (2)} $\Lrho_{\hcoY_{2}*}\sO_{\hcoY_{2}}(tF_{\rho})=\sO_{\hcoY_{3}}$
for $t\geq0$. \end{lem}
\begin{proof}
(1) follows from the relative Kodaira vanishing theorem since $tF_{\rho}-K_{\hcoY_{2}}\equiv_{\hcoY_{3}}(t-5)F_{\rho}$
is $\Lrho_{\hcoY_{2}}$-nef and $\Lrho_{\hcoY_{2}}$-big if $t\leq5$.
(2) is a standard result. 
\end{proof}

In view of Lemma \ref{cla:Tn=4}, we will replace
$\frQ_2$ by $\Lrho_{\hcoY_2}^*\frQ$ if $\sA_2$ or $\sB_2=\frQ_2$,
and write $\sA_2$ or $\sB_2=\Lrho_{\hcoY_2}^*\frQ$.
Then, in any case, $\sA_2$ and $\sB_2$ are the pull-back of 
a locally free sheaf $\sA_3$ and $\sB_3$ on $\hcoY_3$.
Therefore, 
by the Leray spectral sequence for $\Lrho_{\hcoY_{2}}$,
and Proposition \ref{cla:M} and Lemma \ref{cla:van2}, we have the
following:

\begin{lem}
\label{cla:cohY2-Y3} 
\begin{equation}
H^{\bullet}(\hcoY_{2},\sA_2^*\otimes \sB_2(-t))\simeq 
H^{\bullet}(\hcoY_{3},\sA_3^*\otimes \sB_3(-t\det\sQ+tL_{\hcoY_{3}}))\;\;(0\leq t\leq 5).\label{eq:lastvan}\end{equation}
\end{lem}

\noindent {\bf Proof of Theorem \ref{thm:digGvan} for $0\leq t\leq 5$.}
We have only to show the vanishing of
$H^{\bullet}(\hcoY_{3},\sA_3^*\otimes \sB_3(-t\det\sQ+tL_{\hcoY_{3}}))$
for $1\leq t\leq 5$ and $\sA_3, \sB_3=\sO_{\hcoY_3}, \frQ, \sS^*(L), \sQ$,
and for $t=0$ and $(\sA_3,\sB_3)=(\frQ,\sO_{\hcoY_3})$,
$(\sS^*(L),\sO_{\hcoY_3})$,
$(\sQ, \sO_{\hcoY_3})$,
$(\sS^*(L),\frQ)$,
$(\sQ,\frQ)$, or $(\sQ,\sS^*(L))$.

Let $G\simeq\mathrm{G}(3,6)$ be a fiber of $\Lpi_{\hcoY_{3}}$. 
Noting $L_{\hcoY_3}$ and $\frQ$ are the pull-backs of locally free sheaves
on $\hcoY_3$, we see that 
the restriction
of $\sA_3^*\otimes \sB_3(-t\det\sQ+tL_{\hcoY_{3}})$
to $G$ is a direct sum of the following sheaves:
\begin{equation}
\begin{aligned}
&\sO_G(-t)\\
&\sS|_G (-t),
\;\;\;
\sS^*|_G(-t),
\;\;\;
\\
&\sQ|_G (-t),
\;\;\;
\sQ^*|_G(-t),
\;\;\;\\
&\sS|_G\otimes \sQ|_G(-t),
\;\;\; 
\sS^*|_G\otimes \sQ^*|_G (-t),
\;\;\; 
\\
&\sS^*|_G\otimes \sS|_G (-t),
\;\;\;
\sQ^*|_G\otimes \sQ|_G (-t) 
\;\;\;  
\end{aligned}
\label{eq:coh-fiberG}
\end{equation} 

By Theorem \ref{thm:Bott}, all the cohomology
groups of the sheaves in $(\ref{eq:coh-fiberG})$ 
vanish for $0\leq t\leq 5$
except for
$t=0$ and 
$\sS^*|_G$,
$\sQ|_G$,
$\sS^*|_G\otimes \sS|_G$, or
$\sQ^*|_G\otimes \sQ|_G$.
Therefore we may show Theorem 
\ref{thm:digGvan} by the Leray spectral sequence for $\Lpi_{\hcoY_3}$.
$\hfill\square$

\hspace{5pt}

In the subsequent subsections, we will reduce
the proof of Theorem in the remaining cases to
the case where $1\leq t\leq 4$.

\subsection{Case $n=4$ and $6\leq t\leq 9$}

In this case, we have only to show the following:

\begin{prop}
\label{cla:key} For $\sC=\widetilde{\sA}^*\otimes \widetilde{\sB}$ as above, it holds \[
H^{\bullet}(\widetilde{\hcoY},\sC(-t))\simeq H^{13-\bullet}(\widetilde{\hcoY},\sC^{*}(t-10))\]
 for any integer $t$. 
\end{prop}

Then the proof of Theorem \ref{thm:digGvan}
in the case where $6\leq t\leq 9$
is reduced immediately to the case where $1\leq t\leq 4$.
For our proof of the above proposition, we note that each of the cohomology
groups $H^{\bullet}(\widetilde{\hcoY},\sC(-t))$ is Serre dual to
\[
H^{13-\bullet}(\widetilde{\hcoY},\sC^{*}((t-10)M_{\widetilde{\hcoY}}+2F_{\widetilde{\hcoY}})).\]
 Then, from the exact sequence \[
0\to\sC^{*}((t-10)M_{\widetilde{\hcoY}}+(i-1)F_{\widetilde{\hcoY}})\to\sC^{*}((t-10)M_{\widetilde{\hcoY}}+iF_{\widetilde{\hcoY}})\to\sC^{*}((t-10)M_{\widetilde{\hcoY}}+iF_{\widetilde{\hcoY}})|_{F_{\widetilde{\hcoY}}}\to0,\]
 we see that it suffices to show the following vanishings: \begin{equation}
H^{13-\bullet}(F_{\widetilde{\hcoY}},\sA_{i})=0\text{\text{ for }}\ensuremath{i=1,2}\text{\text{ and any}}\;\ensuremath{\bullet},\label{eq:vanF}\end{equation}
 where we set \[
\sA_{i}:=\sC^{*}((t-10)M_{\widetilde{\hcoY}}+iF_{\widetilde{\hcoY}})|_{F_{\widetilde{\hcoY}}}.\]
We evaluate the cohomologies (\ref{eq:vanF}) on the exceptional divisor
$F_{\widetilde{\hcoY}}$ by using the flattening $F^{(3)}\to\widehat{G}'$
of the contraction $F_{\widetilde{\hcoY}}\to G_{\hcoY}$, which we have constructed in \cite[Subsect.~5.5]{Geom}. We will use the notation there freely.
Note that, since the morphism $\widehat{F}\to F_{\widetilde{\hcoY}}$
is finite, and $\widehat{F}$ has only rational singularities by its
construction, the desired vanishings follow from the vanishings of
the cohomology groups of the pull-backs of $\sA_{i}$ on $F^{(3)}$.
Also, since the morphism $F^{(3)}\to\widehat{G}'$ is flat, we have
only to show the vanishing along its fibers. Then, by the upper semi-continuity
of cohomology groups on fibers, it suffices to prove the vanishing
on the fibers $Fib^{(3)}(V_{3},V_{4},V_{4})=A\cup B$ over the points
$([V_{3}];[V_{4}],[V_{4}])\in\widehat{G}'$,
where we refer \cite[Prop.~5.11]{Geom} for the notation.
Note that the restriction of the pull-back of $M_{\widetilde{\hcoY}}$
to the fibers is trivial. Therefore, it suffices to show \begin{equation}
H^{\bullet}(A\cup B,\sC_{A\cup B}^{*}(iF_{A\cup B}))=0\,(i=1,2),\label{eq:H-vanishing-AB}\end{equation}
where $\sC_{A\cup B}$ and $F_{A\cup B}$ are the pull-backs of $\sC$
and $F_{\widetilde{\hcoY}}$ to $A\cup B$, respectively. We recall
$A=\mP(\sO_{\rG(2,V_{4})}\oplus\sU_{\rG(2,V_{4})}(1))\vert_{\rG(2,V_{3})}$
and a natural morphism $A\to\rG(2,V_{3})=\mP(V_{3}^{*})$. Also
recall that $E_{AB}=A\cap B$ is a divisor on $A$. 

In Lemma \ref{lem:FA-FB-Lem} in Appendix \ref{sec:Appendix-D},
we describe several pull-backs to $A$ and $B$ of locally free sheaves on 
$\widetilde{\hcoY}$. Using this we can 
complete our proof of Proposition \ref{cla:key} as follows:

\vspace{5pt}

\noindent {\bf Proof of Proposition $\ref{cla:key}$.}  
It suffices to show the vanishings of (\ref{eq:H-vanishing-AB}).
Tensoring $\sC_{A\cup B}^{*}(iF_{A\cup B})$ with the Mayor-Vietris
sequence, we have \begin{equation}
0\to\sC_{A\cup B}^{*}(iF_{A\cup B})\to\sC_{A}^{*}(iF_{A})\oplus\sC_{B}^{*}(iF_{B})\to\sC_{A\cap B}^{*}(iF_{A\cap B})\to0,\label{eq:MV}\end{equation}
 where $\sC_{A}$, $\sC_{B}$, and $\sC_{A\cap B}$ are the restrictions
of $\sC_{A\cup B}$ to $A$, $B$ and $A\cap B$ respectively, and
$F_{A\cap B}$ is the restriction of $F_{A\cup B}$ to $A\cap B$.
By using Lemma \ref{lem:FA-FB-Lem} (1), it is straightforward to
verify the vanishings of $H^{\bullet}(A,\sC_{A}^{*}(iF_{A}))$. Also,
by Lemma \ref{lem:FA-FB-Lem} (2), we see that the restriction maps
$H^{\bullet}(B,\sC_{B}^{*}(iF_{B}))\to H^{\bullet}(A\cap B,\sC_{A\cap B}^{*}(iF_{A\cap B}))$
are isomorphisms. Then we have the desired vanishings of $H^{\bullet}(A\cup B,\sC_{A\cup B}^{*}(iF_{A\cup B}))$. $\hfill\square$

\subsection{Case $n=3$ and $t=6,7$}
\label{subsection:B3}
For $n=3$, it remains to show
the vanishings of
$H^{\bullet}(\widetilde{\sB}(-6))$ and 
$H^{\bullet}(\widetilde{\sB}(-7))$ for
$\widetilde{\sB}=\sO_{\widetilde{\hcoY}}$,
$\widetilde{\frQ}$, 
$\widetilde{\sS}^*_L$, or
$\widetilde{\sQ}$.
For this, we have only to show the following:

\begin{prop}
\label{cla:key2}
It holds
\[
H^{\bullet}(\widetilde{\hcoY}, \widetilde{\sB}(-t))
\simeq
H^{9-\bullet}(\widetilde{\hcoY}, \widetilde{\sB}^*(t-8))
\]
for any integer $t$.
\end{prop}

\begin{proof}
We can show the assertion in the same way 
to the proof of 
Proposition $\ref{cla:key}$
using Lemma \ref{lem:FA-FB-Lem}.
\end{proof}

\subsection{Calculating $H^0(\widetilde{\hcoY}, \widetilde{\sQ}^*\otimes \widetilde{\sQ}(M))$}
\label{sub:C1}
In the last part of the proof of
\cite[Thm.~8.3.2]{ReyeEnr}, 
we use the following lemma:

\begin{lem}
\label{lem:comple}
Suppose $n=3$.
There exists a unique
$\SL(V)$-equivariant map
$\widetilde{\sQ}^*\boxtimes \sO_{\hchow}\to 
\widetilde{\sQ}^*(M)\boxtimes \sO_{\hchow}(H)$ up to constant.
\end{lem}

\begin{proof}
We compute 
\[
H^0(\widetilde{\sQ}^*\otimes \widetilde{\sQ}(M)\boxtimes
\sO_{\hchow}(H))\simeq
H^0(\widetilde{\hcoY}, \widetilde{\sQ}^*\otimes \widetilde{\sQ}(M))
\otimes
H^0(\hchow, \sO_{\hchow}(H)).
\]
We have 
$H^0(\hchow, \sO_{\hchow}(H))\simeq
{\ft S}^2 V^*$.
To compute
$H^0(\widetilde{\hcoY}, \widetilde{\sQ}^*\otimes \widetilde{\sQ}(M))$,
we tensor $\widetilde{\sQ}^*\otimes \widetilde{\sQ}$ to
the exact sequence
\[
0\to
\sO_{\widetilde{\hcoY}}(M)\to 
\det \widetilde{\sQ} \otimes \wedge^4 V^*
\to \det \widetilde{\sQ}|_{F_{\rho}} \otimes \wedge^4 V^*
\to 0
\]
induced from
Proposition \ref{cla:M}.
By Theorem \ref{thm:Bott} and the Littlewood-Richardson rule,
we have
\begin{eqnarray*}
H^0(\widetilde{\hcoY}, \widetilde{\sQ}^*\otimes \widetilde{\sQ}(\det \widetilde{\sQ}))\otimes \wedge^4 V^*
\simeq
H^0(\widetilde{\hcoY}, \wedge^2 \widetilde{\sQ}\otimes \widetilde{\sQ})\otimes \wedge^4 V^*\simeq\\
(H^0(\widetilde{\hcoY}, \wedge^3 \widetilde{\sQ})\oplus
H^0(\widetilde{\hcoY}, {\ft \Sigma}^{(2,1)} \widetilde{\sQ}))
\otimes \wedge^4 V^*\simeq\\
(\wedge^3 (\wedge^2 V)\oplus {\ft \Sigma}^{(2,1,0,0,0,0)} \wedge^2 V)
\otimes \wedge^4 V^*.
\end{eqnarray*}
By the plethysm of the Schur functors
\[
\wedge^3 (\wedge^2 V)\simeq
{\ft \Sigma}^{(3,1,1,1)} V\oplus 
{\ft \Sigma}^{(2,2,2,0)} V,
\]
\[
{\ft \Sigma}^{(2,1,0,0,0,0)} \wedge^2 V\simeq
{\ft \Sigma}^{(2,2,1,1)} V
\oplus
{\ft \Sigma}^{(3,2,1,0)} V,
\]
we obtain
\begin{align}
\label{eqn:Schur1}
&H^0(\widetilde{\hcoY}, \widetilde{\sQ}^*\otimes \widetilde{\sQ}(\det \widetilde{\sQ}))\otimes \wedge^4 V^*
\simeq\\
&{\ft \Sigma}^{(2,0,0,0)} V\oplus 
{\ft \Sigma}^{(1,1,1,-1)} V\oplus
{\ft \Sigma}^{(1,1,0,0)} V
\oplus
{\ft \Sigma}^{(2,1,0,-1)} V.\nonumber
\end{align}
By Proposition \ref{cla:det},
we have
\begin{align*}
&H^0(F_{\rho},\widetilde{\sQ}^*\otimes \widetilde{\sQ}(\det \widetilde{\sQ})|_{F_{\rho}}) \otimes \wedge^4 V^*\simeq\\
&H^0(\mP(V), (\Omega^1_{\mP(V)}(1))^*\otimes \Omega^1_{\mP(V)}(1)\otimes
\sO_{\mP(V)}(2))\otimes \wedge^4 V.
\end{align*} 
By Theorem \ref{thm:Bott} and the Littlewood-Richardson rule,
we obtain
\begin{eqnarray}
\label{eqn:Schur2}
H^0(F_{\rho},\widetilde{\sQ}^*\otimes \widetilde{\sQ}(\det \widetilde{\sQ})|_{F_{\rho}}) \otimes \wedge^4 V^*\simeq
{\ft \Sigma}^{(1,1,1,-1)} V
\oplus
{\ft \Sigma}^{(2,1,0,-1)} V.
\end{eqnarray} 
Since the identity of $\Hom (\widetilde{\sQ},\widetilde{\sQ})$
induces
that of $\Hom (\Omega^1_{\mP(V)}(2), \Omega^1_{\mP(V)}(2))$,
the component 
${\ft \Sigma}^{(1,1,1,-1)} V$ in (\ref{eqn:Schur1}) is mapped 
isomorphically in (\ref{eqn:Schur2}).
Therefore,
$H^0(\widetilde{\hcoY}, \widetilde{\sQ}^*\otimes \widetilde{\sQ}(M))$
consists of
at most
${\ft \Sigma}^{(2,0,0,0)} V\simeq {\ft S}^2 V$,
${\ft \Sigma}^{(1,1,0,0)} V$, and
${\ft \Sigma}^{(2,1,0,-1)} V$
(it seems that
$H^0(\widetilde{\hcoY}, \widetilde{\sQ}^*\otimes \widetilde{\sQ}(M))\simeq
{\ft \Sigma}^{(2,0,0,0)} V\oplus {\ft \Sigma}^{(1,1,0,0)} V$
but we do not need to prove this).
Hence
$\SL(V)$-invariant sections of
$H^0(\widetilde{\hcoY}, \widetilde{\sQ}^*\otimes \widetilde{\sQ}(M))
\otimes
H^0(\hchow, \sO_{\hchow}(H))$
come from the identity element of 
${\ft S}^2 V\otimes {\ft S}^2 V^*$
up to constant.
\end{proof}

\subsection{Lefschetz collection in $\sD^{b}(\widetilde{\hcoY})$}

It is straightforward to obtain the following result from Theorem \ref{thm:digGvan}.

\begin{cor} \label{thm:Gvann=3} $\;$
Suppose $n=3$.
Let $\Lambda:=\left\{ 3,2,1_{a},1_{b}\right\}$
be an ordered set $(\Lambda,\prec)$. 
Define \[
(\sE_{\alpha})_{\alpha\in\Lambda}:=
(\sE_{3},\sE_{2},\sE_{1a},\sE_{1b})=
(\sO_{\widetilde{\hcoY}},\, \widetilde{\frQ},\, 
\widetilde{\sS}_{L}^{*},\,
\widetilde{\sQ}
)
\]
be an ordered collection of sheaves.
Set \[
\sD_{{\hcoY}}:=\langle\sE_{3},\sE_{2},\sE_{1a},\sE_{1b}\rangle\subset\sD^{b}(\widetilde{\hcoY}).\]
 Then \[
\sD_{{\hcoY}},\sD_{{\hcoY}}(1),\dots,\sD_{{\hcoY}}(5), \sO_{\widetilde{\hcoY}}(6), \sO_{\widetilde{\hcoY}}(7)
\]
 is a Lefschetz collection, 
where $(t)$ represents the twist by the sheaf $\sO_{\widetilde{\hcoY}}(tM)$. 
\end{cor}

\begin{rem}
By similar calculations to show Theorem \ref{thm:digGvan},
we obtain the following:

\begin{enumerate}[(1)]
\item
The ordered collection of sheaves
$(\sE_{3},\sE_{2},\sE_{1a},\sE_{1b})$
is a strongly exceptional collection in $\sD^b(\widetilde{\hcoY})$.
\item
$\Hom$'s of the sheaves in the above collection are 
given by the following diagram$:$ 
\[
\begin{matrix}
\begin{xy}
(5,-12)*+{\circ}="c3",
(20,-12)*+{\circ}="c2",
(35,0)*+{\circ}="c1a",
(35,-24)*+{\circ}="c1b",
(0,-12)*{\mathcal{E}_3},
(20,-15)*{\mathcal{E}_2},
(38,3)*{\mathcal{E}_{1a}},
(38,-27)*{\mathcal{E}_{1b}},
(18,1)*{\wedge^2V},
(18,-25)*{\wedge^2 V},
(30,-8)*{V},
(15,-9)*{V},
(30,-16)*{V},
\ar "c3";"c2"
\ar "c2";"c1a"
\ar "c2";"c1b"
\ar @/^1pc/@{->} "c3";"c1a"
\ar @/_1pc/@{->} "c3";"c1b"
\end{xy}
\end{matrix}
%\label{eq:quiver2}
\]

%\begin{equation}
%\xyQuiverII\label{eq:quiver2}\end{equation}
 
\end{enumerate}
\end{rem}

\vspace{5pt}

In case $n=4$,
the following Lefschetz collection is suitable for our purpose (see \cite{HoTa3}).

\begin{cor}
\label{thm:Gvan} $\;$
Suppose $n=4$. Let $\Lambda:=\left\{ 3,2,1_{a},1_{b}\right\} $
be an ordered set $(\Lambda,\prec)$. 
Define \[
(\sE_{\alpha})_{\alpha\in\Lambda}:=(\sE_{3},\sE_{2},\sE_{1_{a}},\sE_{1_{b}})=(\widetilde{\sS}_{L},\,\widetilde{\frQ}^{*},\,\sO_{\widetilde{\hcoY}},\,\widetilde{\sQ}^{*}(M))\]
 be an ordered collection of sheaves on $\widetilde{\hcoY}$. 
Set \[
\sD_{\widetilde{\hcoY}}:=\langle\sE_{3},\sE_{2},\sE_{1_{a}},\sE_{1_{b}}\rangle\subset\sD^{b}(\widetilde{\hcoY}).\]
 Then \[
\sD_{\widetilde{\hcoY}},\sD_{\widetilde{\hcoY}}(1),\dots,\sD_{\widetilde{\hcoY}}(9)\]
 is a Lefschetz collection,
where $(t)$ represents the twist by the sheaf $\sO_{\widetilde{\hcoY}}(tM)$. 
\end{cor}

\begin{proof}
Writing 
$\sD_{\widetilde{\hcoY}},\sD_{\widetilde{\hcoY}}(1),\dots,\sD_{\widetilde{\hcoY}}(9)$ explicitly, we obtain
\begin{eqnarray*}
\widetilde{\sS}_{L},\,\widetilde{\frQ}^{*},\,\sO_{\widetilde{\hcoY}},\\
\widetilde{\sQ}^{*}(M),\, \widetilde{\sS}_{L}(M),\,\widetilde{\frQ}^{*}(M),\,\sO_{\widetilde{\hcoY}}(M),\\
\cdots,\\
\widetilde{\sQ}^{*}(9M),\, \widetilde{\sS}_{L}(9M),\,\widetilde{\frQ}^{*}(9M),\,\sO_{\widetilde{\hcoY}}(9M),\\
\widetilde{\sQ}^{*}(10M).
\end{eqnarray*}
Let $\sC$ be any sheaf in this sequence except
$\widetilde{\sQ}^{*}(10M)$.
Then, by Proposition \ref{cla:key},
we have 
$\Hom^{\bullet}(\widetilde{\sQ}^{*}(10M), \sC)\simeq
\Hom^{\bullet}(\sC, \widetilde{\sQ}^{*})$.
Hence we may replace the sequence
by
$\sD', \sD'(1),\dots, \sD'(9)$
with
$\sD'=\langle 
\widetilde{\sQ}^{*},\, \widetilde{\sS}_{L},\,\widetilde{\frQ}^{*},\,\sO_{\widetilde{\hcoY}}\rangle$.
Now the assertion follows immediately from 
Theorem \ref{thm:digGvan}.
\end{proof}

\begin{rem}
By similar calculations to show Theorem \ref{thm:digGvan},
we obtain the following (see \cite{Arxiv}):
\begin{enumerate}[(1)]
\item
$(\sE_{\alpha})_{\alpha\in\Lambda}$ is a strongly exceptional collection.

\item 
$\Hom$'s of the sheaves in the above collection are 
given by the following diagram$:$ 
\def\xyQuiverII{ 
\begin{matrix} 
\begin{xy} 
(5,-12)*+{\circ}="c3", 
(20,-12)*+{\circ}="c2", 
(35,0)*+{\circ}="c1a", 
(35,-24)*+{\circ}="c1b", 
(0,-12)*{\mathcal{E}_3}, 
(20,-15)*{\mathcal{E}_2}, 
(38,3)*{\mathcal{E}_{1_a}}, 
(38,-27)*{\mathcal{E}_{1_b}}, 
(18,1)*{\wedge^2V}, (18,-25)*{{\ft S}^2V}, 
(30,-8)*{V}, (15,-9)*{V}, 
(30,-16)*{V}, 
\ar "c3";"c2" 
\ar "c2";"c1a" 
\ar "c2";"c1b"
\ar @/^1pc/@{->} "c3";"c1a" 
\ar @/_1pc/@{->} "c3";"c1b" 
\end{xy} 
\end{matrix} } \begin{equation}
\xyQuiverII\label{eq:quiver2}\end{equation}
 
\end{enumerate}
\end{rem}

\begin{rem}
 Since the singularity of the double symmetroid $\hcoY$ is complicated,
\cite[Theorem 1]{Lef} seems to be difficult to apply
for the resolution $\widetilde{\hcoY}\to\hcoY$ to obtain a categorical
resolution of $D^{b}(\hcoY)$. However, we expect
that the Lefschetz collection $\sD_{\widetilde{\hcoY}},$
$\sD_{\widetilde{\hcoY}}(1),\dots,\sD_{\widetilde{\hcoY}}(9)$
gives a Lefschetz
decomposition of a strongly crepant categorical resolution, if exists, of $\sD^{b}(\hcoY)$. 

\end{rem}

\appendix
%dummy comment inserted by tex2lyx to ensure that this paragraph is not empty
\section{{\bf Pull-backs of sheaves to the flattening of $F_{\widetilde{\hcoY}}\to G_{\hcoY}$}}
\label{sec:Appendix-D}
In this section, we consider the situation as in \cite[Subsect.~5.5]{Geom} and use the notation there freely.

Here we fix $V_{n-1}$ and $V_{n}$, and consider sheaves on the fiber
$Fib^{(3)}(V_{n-1},V_{n},V_{n})=A\cup B$ of $F^{(3)}\to\hat{G}'$. 

\begin{lem}
\label{lem:FA-FB-Lem}Denote by $H_{A}$ the pull-back on $A$ of
$\sO_{\mP(V_{n-1}^*)}(1)$. We denote by $\sE_{A}$ and $\sE_{B}$ the
pull-backs of a locally free sheaf $\sE$ on $\widetilde{\hcoY}$
to $A$ and $B$, respectively. In particular, $F_{A}$ and $F_{B}$
stand for the pull-backs of $F_{\widetilde{\hcoY}}$
to $A$ and $B$. Then we have the following isomorphisms\,$:$

\begin{myitem2} \item[$(1)$] $F_{A}\sim -(E_{AB}+2H_{A})$, $(\widetilde{\sS}_{L}^{*})_{A}\simeq\widetilde{\sQ}_{A}\simeq\sO_{A}\oplus\sV$,
and $\widetilde{\frQ}_{A}\simeq\sO_{A}^{\oplus2}\oplus\sV$ with a locally
free sheaf $\sV$ given by a unique non-split extension \[
0\to\sO_{A}(H_{A}+E_{AB})\to\sV\to\sO_{A}(H_{A})\to0.\]

\item[$(2)$] $F_{B}\simeq p_{B}^{*}\sO_{\rG(n-3,V_{n-1})}(-1)$, $(\widetilde{\sS}_{L}^{*})_{B}\simeq\widetilde{\sQ}_{B}\simeq\sO_{B}\oplus p_{B}^{*}\frQ_{V_{n-1}}$,
and $\widetilde{\frQ}_{B}\simeq\sO_{B}^{\oplus2}\oplus p_{B}^{*}\frQ_{V_{n-1}}$,
where $\frQ_{V_{n-1}}$ is the universal quotient bundle on
$\rG(n-3,V_{n-1})$, and 
$p_{B}\colon B\to\rG(n-3,V_{n-1})$ is given in \cite[Prop.~5.10 and 5.11]{Geom}.

\end{myitem2}\end{lem}
\begin{proof}
Let $G$ be the exceptional divisor for $\widehat{A}\to A$
and $L_{\widehat{A}}$, $H_{\widehat{A}}$, and $\widehat{E}_{AB}$ the pull-backs
on $\widehat{A}$ of $\sO_{\rG(n-3,V_{n-1})}(1)$, $H_{A}$ and $E_{AB}$, respectively.

\noindent\textbf{Step 1. $\widehat{E}_{AB}+2H_{\widehat{A}}-G=L_{\widehat{A}}$.}

As in \cite[Prop.~5.13]{Geom}, we denote by $s_A\subset A$ the locus of $\rho$-conics, which is a section associated to an injection 
to an injection 
$\sO_{\mP(V_{n-1}^{*})}\to \sO_{\mP(V_{n-1}^{*})}\oplus T_{\mP(V_{n-1}^*)}$. 
Let $\sI$ be the ideal sheaf of $s_A$ in $A$, and
consider the exact sequence
\[
0\to \sO_A(E_{AB}+2H_A)\otimes \sI\to \sO_A(E_{AB}+2H_A)\to \sO_{s_A}(2H_A)\to 0,
\]
where the last term is obtained since $s_A\cap E_{AB}=\emptyset$.
Let $\pi_A\colon A\to \mP(V_{n-1}^*)$ be the natural morphism.
Then ${\pi_A}_*(\sO_A(E_{AB}+2H_A)\otimes \sI)\simeq 
\Omega^1_{\mP(V_{n-1}^*)}(2)\simeq \wedge^{n-3}T_{\mP(V_{n-1}^*)}$.
From this, we see that the natural map 
\[
H^0(\sO_A(E_{AB}+2H_A)\otimes \sI)\otimes \sO_A \to \sO_A(E_{AB}+2H_A)\otimes \sI
\]
is surjective and 
$H^0(\sO_A(E_{AB}+2H_A)\otimes \sI)\simeq H^0(\wedge^{n-3}T_{\mP(V_{n-1}^*)})
\simeq \wedge^{n-3} V_{n-1}^*$.
This is equivalent to that 
$|\widehat{E}_{AB}+2H_{\widehat{A}}-G|$ is base point free
and it defines a morphism $\Phi\colon \widehat{A}\to \mP(\wedge^{n-3} V_{n-1})$. 
$\Phi$ factors through $A_{\hcoY_2}$ since it contracts $E_{AB}$.
Let $\Phi'\colon A_{\hcoY_2}\to \mP(\wedge^{n-3} V_{n-1})$ be the induced morphism.
$\Phi'$ does not coincide with $A_{\hcoY_2}\to A_{\widetilde{\hcoY}}$
since the latter contracts the image of $G$.
Therefore $\Phi'$ induces the quadric fibration $A_{\hcoY_2}\to \rG(n-3,V_{n-1})$. In particular, we have
$\widehat{E}_{AB}+2H_{\widehat{A}}-G=L_{\widehat{A}}$ as desired.

\noindent\textbf{Step 2. $\det(\widetilde{\sS}_{L}^{*})_{A}=\det\widetilde{\sQ}_{A}=E_{AB}+2H_{A}$.}

 By (\ref{eq:detdiff}), we have only to determine $\det\widetilde{\sQ}_{A}$.
Note
that $G$ coincides with the pull-back of $F_{\rho}\vert_{\hcoY_{2}}$,
where $F_{\rho}$ is the exceptional divisor of $\hcoY_{2}\to\widetilde{\hcoY}$.
Therefore, by Proposition \ref{cla:M}, we have $\det\widetilde{\sQ}_{A}=G+L_{\widehat{A}}$
since $M_{\widetilde{\hcoY}}$ is trivial on a fiber of $F_{\widetilde{\hcoY}}\to G_{\hcoY}$.
Now the assertion follows from Step 1.

\noindent\textbf{Step 3. $F_{A}=-(E_{AB}+2H_{A})$.}

By Proposition \ref{prop:F'}, we have
$F'_{\widetilde{\hcoY}}=2M_{\hcoY_2}-L_{\hcoY_2}-F_{\rho}$.
Since $G$ coincides with the pull-back of $F_{\rho}\vert_{\hcoY_{2}}$
and $M_{\widetilde{\hcoY}}$ is trivial on a fiber of $F_{\widetilde{\hcoY}}\to G_{\hcoY}$,
it holds that 
$F_{\widehat{A}}=-(L_{\widehat{A}}+G)$. 
Therefore the assertion follows
from Step 1. 

\noindent\textbf{Step 4. $(\widetilde{\sS}_{L}^{*})_{A}\simeq\widetilde{\sQ}_{A}\simeq\sO_{A}\oplus\sV$.}

 We investigate the restriction of the universal exact sequence
(\ref{eq:univ}) on $A_{\hcoY_{3}}$. Let $\sS_{A_{\hcoY_{3}}}$ and
$\sQ_{A_{\hcoY_{3}}}$ be the restrictions of $\sS$ and $\sQ$, respectively.
Then we obtain \begin{equation}
0\to\sS_{A_{\hcoY_{3}}}\to\Lpi_{A_{\hcoY_{3}}}^{\;*}(\wedge^2 \frQ|_{\rG(n-3,V_{n-1})})\to\sQ_{A_{\hcoY_{3}}}\to0.\label{eq:resuniv}\end{equation}
$\frQ|_{\rG(n-3,V_{n-1})}$ splits as 
\[
\frQ|_{\rG(n-3,V_{n-1})}\simeq \frQ_{V_{n-1}}\oplus V/V_{n-1}\otimes\sO_{{\rG(n-3,V_{n-1})}}.
\]
Therefore we have the following isomorphisms: 
\begin{eqnarray}
\wedge^{2}(\frQ|_{\rG(n-3,V_{n-1})})\simeq\label{eq:resT-1}\qquad\qquad\qquad\qquad\qquad\qquad\qquad\qquad\qquad\qquad\\
\sO_{{\rG(n-3,V_{n-1})}}(1)\oplus\Big(\frQ_{V_{n-1}} \otimes (V/V_{n-1}) \Big)
\oplus\wedge^{2}(V/V_{n-1})\otimes\sO_{{\rG(n-3,V_{n-1})}}.\nonumber \end{eqnarray}
Let $[V_{n-3}]$ be a point of $\rG(n-3,V_{n-1})$.
(\ref{eq:resT-1}) means fiberwise 
\begin{equation}
\label{eq:fiber}
\wedge^{2}(V/V_{n-3})\simeq
\wedge^2 (V_{n-1}/V_{n-3})\oplus (V_{n-1}/V_{n-3})\otimes (V/V_{n-1})
\oplus
\wedge^{2}(V/V_{n-1}).\end{equation}
Now we recall \cite[Rem.~5.14]{Geom}.
Let $\Gamma$ be the fiber of $A_{\hcoY_3}\to \rG(n-3,V_{n-1})$ over $[V_{n-3}]$.
The vertex of the quadric cone $\Gamma$ corresponds to
the $\sigma$-plane ${\rm P}_{V_n/V_{n-3}}=\{\mC^2\subset V_n/V_{n-3}\}$.
Points $[{\rm P}_{V_{n-2}/V_{n-3}}]$ which correspond to $\rho$-planes and are contained in $\Gamma$ satisfy $V_{n-3}\subset V_{n-2}$.
Since $\Gamma$ is the cone over the Veronese curve $v_2(\mP(V_{n-1}/V_{n-3}))$,
it is swept out by lines joining $[{\rm P}_{V_n/V_{n-3}}]$ and $[{\rm P}_{V_{n-2}/V_{n-3}}]$ such that
$V_{n-3}\subset V_{n-2}\subset V_{n-1}$.
A line in
$\rG(3,\wedge^{2}(V/V_{n-3}))$ is of the form
$\{W_2\subset \mC^3\subset W_4\}$ with some $W_i\simeq \mC^i$ $(i=2,4)$.
We take a basis ${\bm e}_1,\dots, {\bm e}_4$ of $V/V_{n-3}$
such that $V_{n-2}/V_{n-3}=\langle {\bm e}_1\rangle$,
$V_{n-1}/V_{n-3}=\langle {\bm e}_1, {\bm e}_2\rangle$, and 
$V_{n}/V_{n-3}=\langle {\bm e}_1, {\bm e}_2, {\bm e}_3\rangle$.
For the line joining 
$[{\rm P}_{V_n/V_{n-3}}]$ and $[{\rm P}_{V_{n-2}/V_{n-3}}]$,
it is easy to see that \[
W_2=\langle {\bm e}_1\wedge {\bm e}_2, {\bm e}_1\wedge {\bm e}_3\rangle\simeq
 \wedge^2 (V_{n-1}/V_{n-3})\oplus (V_{n-2}/V_{n-3})\otimes (V_n/V_{n-1}),
\]
\[
W_4=\langle {\bm e}_1\wedge {\bm e}_2, {\bm e}_1\wedge {\bm e}_3,
{\bm e}_1\wedge {\bm e}_4, {\bm e}_2\wedge {\bm e}_3\rangle
\subset 
\wedge^2 (V_{n-1}/V_{n-3})\oplus (V_{n-1}/V_{n-3})\otimes (V/V_{n-1}),
\]
and
\[
(\wedge^2 (V_{n-1}/V_{n-3})\oplus (V_{n-1}/V_{n-3})\otimes (V/V_{n-1}))/W_4\simeq \langle {\bm e}_2\wedge {\bm e}_4\rangle \simeq V_{n-1}/V_{n-2}\otimes V/V_n. 
\]

As for $\sS_{\widehat{A}}$, these imply the following:
\begin{itemize}
\item
We can see that $\sS_{\widehat{A}}$ contains
the line bundle $L_{\widehat{A}}$
with fiber $\wedge^2 (V_{n-1}/V_{n-3})$
as a direct summand.
Hence, let us write $\sS_{\widehat{A}}
=\sS'_{\widehat{A}}\oplus L_{\widehat{A}}$
with a locally free sheaf $\sS'_{\widehat{A}}$ of rank two on $\widehat{A}$.
\item
$\sS'_{\widehat{A}}$ contains a sub line bundle
with fiber 
$(V_{n-2}/V_{n-3})\otimes (V_n/V_{n-1})$, which is isomorphic to
$-H_{\widehat{A}}+L_{\widehat{A}}$.
\end{itemize}
Therefore, by Step 2, we obtain
\[
0\to\sO_{\widehat{A}}(-H_{\widehat{A}})\to
{\sS}'_{\widehat{A}}(-L_{\widehat{A}})\to\sO_{\widehat{A}}(-H_{\widehat{A}}-\widehat{E}_{AB})\to0.\]
Since all the terms of the exact sequence
are the pull-backs of locally free sheaves on $A$,
the dual of this exact sequence descends to
\[
0\to\sO_{{A}}(H_{{A}}+{E}_{AB})\to(\sS'_L)^*_{{A}}\to
\sO_{{A}}(H_{{A}})\to0,\]
where $(\sS'_L)^*_{{A}}$ is the locally free sheaf on $A$ such that its bull-back on $\widehat{A}$ is equal to $({\sS}'_{\widehat{A}})^*(L_{\widehat{A}})$.
This sequence
does not split since $({\sS}'_{L})^*_{{A}}$ comes from a
locally free sheaf on $A_{\widetilde{\hcoY}}$ while $H_{A}$ does
not. From
\begin{align*}
 & \Ext^{1}(\sO_{A}(H_{{A}}),\sO_{A}(H_{{A}}+E_{AB}))\simeq\\
 & H^{1}(A,\sO_{A}(E_{AB}))\simeq H^{1}(\mP(V_{n-1}^{*}),\sO_{\mP(V_{n-1}^{*})}\oplus\Omega_{\mP(V_{n-1}^{*})}^{1})\simeq\mC,\end{align*}
such a nonsplit extension is unique, which we denote by $\sV$.
Thus we obtain
$(\widetilde{\sS}_{L}^{*})_{A}\simeq\sV\oplus\sO_{A}$, with a locally
free sheaf $\sV$ as described in (1).

Similarly, as for $\sQ_{\widehat{A}}$, the above facts imply the following:
\begin{itemize}
\item
We can see that $\sQ_{\widehat{A}}$ contains
the line bundle $\sO_{\widehat{A}}$ with fiber $\wedge^{2}(V/V_{n-1})$
as a direct summand.
Hence, let us write $\sQ_{\widehat{A}}
=\sQ'_{\widehat{A}}\oplus \sO_{\widehat{A}}$
with a locally free sheaf $\sQ'_{\widehat{A}}$ of rank two on $\widehat{A}$.
\item
$\sQ'_{\widehat{A}}$ have a quotient line bundle
with fiber 
$V_{n-1}/V_{n-2}\otimes V/V_n$, which is isomorphic to
$H_{\widehat{A}}$.
\end{itemize}
Therefore, by Step 2, we obtain
\[
0\to\sO(H_{\widehat{A}}+\widehat{E}_{AB})\to{\sQ}'_{\widehat{A}}\to\sO(H_{\widehat{A}})\to0.\]
In a similar way to determine $(\widetilde{\sS}_{L}^{*})_{A}$,
we may obtain  
$\sQ_{A}\simeq\sV\oplus\sO_{A}$ as desired.

\noindent\textbf{Step 5. $\widetilde{\frQ}_{A}\simeq\sO_{A}^{\oplus2}\oplus\sV$.}

By \cite[Rem.~5.14]{Geom},
$\Prt_{\rho}\cap A_{\hcoY_{3}}\simeq\mP(\frQ_{V_{n-1}})\simeq \rF(n-3,n-2,V_{n-1}).$
Restricting (\ref{eq:exact-seq-T2})
to ${A_{{\hcoY}_{2}}}$, 
we obtain 
\begin{equation}
0\to\frR_{A_{\hcoY_{2}}}\to\Lpi_{{A_{\hcoY_{2}}}}^{\;*}\frQ_{V_{n-1}}^*\oplus\sO_{A_{\hcoY_2}}^{\oplus2}\to\iota_{*}\sO_{\mP(\frQ\,_{V_{n-1}})}(1)\to0,\label{eq:E_a'res}
\end{equation}
 where we set $\frR_{A_{\hcoY_{2}}}=\frR_{2}\vert_{A_{\hcoY_{2}}},
\Lpi_{A_{\hcoY_{2}}}=\Lpi_{\hcoY_{2}}\vert_{A_{\hcoY_{2}}}$,
$\iota:\mP(\frQ_{V_{n-1}})\hookrightarrow A_{\hcoY_{2}}$
and use $\sR_{2}/\sR_{1}\simeq\sO_{\mP(\frQ\ )}(-1)$. Since $\Hom(\sO_{A_{\hcoY_2}},\sO_{\mP(\frQ\,_{V_{n-1}})}(1))=H^{0}(\frQ_{V_{n-1}}^*)$
$=0$, we have the decomposition $\frR_{A_{\hcoY_{2}}}\simeq\sO_{A_{\hcoY_{2}}}^{\oplus2}\oplus\sV'$
and also 
\begin{equation}
0\to\sV'\to\Lpi_{{A_{\hcoY_{2}}}}^{\;*}\frQ_{V_{n-1}}^*\to\iota_{*}\sO_{\mP(\frQ\,_{V_{n-1}})}(1)\to0.\label{eq:V'}
\end{equation}

Note that
$\Lpi_{{\widehat{A}}}$ is a $\mF_2$-fibration and
it decomposes as $\widehat{A}\to \mP(\frQ_{V_{n-1}})\to \rG(n-3,V_{n-1})$.
On $\mP(\frQ_{V_{n-1}})\simeq \rF(n-3, n-2,V_{n-1})$,
we have a natural exact sequence
\[
0\to \sO_{\mP(\frQ\,_{V_{n-1}})}(-H_{\widehat{A}})\to \Lpi^*\frQ_{V_{n-1}}^*\to 
\sO_{\mP(\frQ\,_{V_{n-1}})}(1)\to0.
\]
Pulling back on $\widehat{A}$,
we obtain
$0\to \sO_{\widehat{A}}(-H_{\widehat{A}})\to \Lpi_{\widehat{A}}^*\frQ_{V_{n-1}}^*\to
\sO_{\widehat{A}}(H_{\widehat{A}}-L_{\widehat{A}})\to 0$
such that the composite of 
$\sO_{\widehat{A}}(-H_{\widehat{A}})\to \Lpi_{\widehat{A}}^*\frQ_{V_{n-1}}^*$ with 
$\Lpi_{{A_{\hcoY_{2}}}}^{\;*}\frQ_{V_{n-1}}^*\to\iota_{*}\sO_{\mP(\frQ\,_{V_{n-1}})}(1)$ is a $0$-map.
Therefore it induces an injection 
$\sO_{\widehat{A}}(-H_{\widehat{A}})\to \sV'$ with a locally free cokernel.
Computing the determinants, we obtain
\[
0\to \sO_{\widehat{A}}(-H_{\widehat{A}})\to \sV'\to\sO_{\widehat{A}}(H_{\widehat{A}}-L_{\widehat{A}}-G)\to 0. 
\]
Since $H_{\widehat{A}}-L_{\widehat{A}}-G=-H_{\widehat{A}}-\widehat{E}_{AB}$ by Step 1, we obtain 
$\widetilde{\frQ}_{A}\simeq\sO_{A}^{\oplus2}\oplus\sV$
in a similar way to determine $(\widetilde{\sS}_{L}^{*})_{A}$.

\noindent\textbf{Step 6. $F_{B}$, $(\widetilde{\sS}_{L}^{*})_{B}$, $\widetilde{\sQ}_{B}$
and $\widetilde{\frQ}_{B}$.}

 By \cite[Prop.~5.10 and 5.11]{Geom}, the image of $B$ on $F_{\widetilde{\hcoY}}$
is the $\rG(n-3, V_{n-1})$ in $A_{\widetilde{\hcoY}}$. Therefore, $F_{B}$,
$(\widetilde{\sS}_{L}^{*})_{B}$, $\widetilde{\sQ}_{B}$ and $\widetilde{\frQ}_{B}$,
respectively, are the pull-backs of the restrictions of $F_{\widetilde{\hcoY}}$,
$\widetilde{\sS}_{L}^{*}$ $\widetilde{\sQ}$, and $\widetilde{\frQ}$
to $\rG(n-3, V_{n-1})$. Since $F_{A}|_{E_{AB}}\simeq-(E_{AB}+2H_{A})|_{E_{AB}}$
by Step 3, and this is the pull-back of $\sO_{\rG(n-3, V_{n-1})}(-1)$, 
we have $F_{B}=p_{B}^{*}\sO_{\rG(n-3, V_{n-1})}(-1)$.
Also, since $\widetilde{\frQ}_{A}\simeq\sO_{A}\oplus(\widetilde{\sS}_{L}^{*})_{A}\simeq\sO_{A}\oplus\widetilde{\sQ}_{A}$
as above, we have $\widetilde{\frQ}_{B}\simeq\sO_{B}\oplus(\widetilde{\sS}_{L}^{*})_{B}\simeq\sO_{B}\oplus\widetilde{\sQ}_{B}$.
Thus we have only to determine $\widetilde{\frQ}_{B}$. 
Since $\rG(n-3, V_{n-1})$ is contained in the locus of $\sigma$-plane,
it is disjoint from the locus $G_{\rho}$ of $\rho$-conics.
Therefore, by (\ref{eq:Eb}),
we have $\widetilde{\frQ}_{B}\simeq p_{B}^{*}(\frQ|_{\rG(n-3, V_{n-1})})\simeq\sO_{B}^{\oplus2}\oplus p_{B}^{*}(\frQ_{V_{n-1}})$.
\end{proof}

$\;$

\vspace{1cm}
\noindent {\footnotesize Department of Mathematics, Gakushuin University, 
Toshima-ku,Tokyo 171-8588,$\,$Japan }{\footnotesize \par}

\noindent {\footnotesize e-mail address: hosono@math.gakushuin.ac.jp} 

\vspace{5pt}

\noindent {\footnotesize Graduate School of Mathematical Sciences,
University of Tokyo, Meguro-ku,Tokyo 153-8914,$\,$Japan }{\footnotesize \par}

\noindent {\footnotesize e-mail address: takagi@ms.u-tokyo.ac.jp} 
\end{document}